\documentclass[11pt]{article}

\usepackage{amsmath,amsfonts,amsthm,amscd,amssymb,graphicx}

\usepackage{mathabx}

\usepackage[utf8]{inputenc}

\usepackage{subfigure}
\usepackage{xcolor}

\numberwithin{equation}{section}

\usepackage{enumitem}
\usepackage{hyperref}
\usepackage{thmtools}
\usepackage{cleveref}

\usepackage{cases}

\newtheorem{theorem}{Theorem}[section]

\newtheorem{lemma}[theorem]{Lemma}

\newtheorem{proposition}[theorem]{Proposition}

\newtheorem{remark}[theorem]{Remark}

\newcommand{\RR}{\mathbb{R}}

\def\cX{\mathcal{X}}
\def\FcX{\widehat{\cX}}
\def\FcP{\widehat{\cP}}

\newcommand{\cH}{\mathcal{H}}

\newcommand{\FF}{{\widehat F}}

\def\cI{\mathcal{I}}
\def\cP{\mathcal{P}}
\def\cL{\mathcal{L}}
\def\hw{\widehat{w}}
\def\hgamma{\widehat{\gamma}}

\def\cN{\mathcal{N}}

\def\FK{\widehat{K}}

\def\Frho{\widehat{\rho}}

\def\FS{\widehat{S}}

\def\FG{\widehat{G}}

\def\FA{\widehat{A}}
\def\FB{\widehat{B}}

\def\FI{\widehat{I}}

\def\cJ{\mathcal{J}}

\def\FA{\widehat{A}}

\def\Frho{\widehat{\rho}}

\def\Fmu{\widehat{\mu}}

\def\cD{\mathcal{D}}
\def\tlambda{\tilde{\lambda}}
\def\ttau{\tilde{\tau}}

\def\gammaeq{\gamma_f}
\def\Fgammaeq{\widehat{\gamma}_f}
\def\Fgamma{\widehat{\gamma}}

\makeatletter
\newcommand{\labitem}[2]{%
	\def\@itemlabel{\textbf{#1}}
	\item
	\def\@currentlabel{#1}\label{#2}}
\makeatother

\begin{document}

\title{Asymptotic Stability of Hartree--Fock Homogenous Equilibria in $\mathbb{R}^d$} 

\author{Toan T. Nguyen\footnotemark[1]
	\and Chanjin You\footnotemark[1]
}

\maketitle

\footnotetext[1]{Penn State University, Department of Mathematics, State College, PA 16802. Emails: nguyen@math.psu.edu, cby5175@psu.edu. The research is supported in part by the NSF under grant DMS-2349981.}

\maketitle

\begin{abstract}

In this paper, we establish nonlinear Landau damping and asymptotic stability of a large class of translation-invariant steady solutions to the time-dependent Hartree--Fock equations in the presence of an {\em off-diagonal exchange operator}, which arises naturally in the meanfield theory of a large fermionic system, in the whole space $\mathbb{R}^d$, $d\ge 3$. Despite being a sub-order operator, the inclusion of the exchange term disturbs the classical Schr\"odinger dispersion and causes a complex linear response from the background electrons to the space density whose dispersion relation is no longer a Fourier multiplier as in the classical Vlasov and Hartree theory. In addition, the group velocity of each elementary waves involves a mixture of all other Fourier modes, leading to delicate {\em momentum-dependent echo resonances}. To overcome the issues, we develop a nonlinear iterative scheme that relies on a detailed resolvent analysis, makes use of a transport type dispersion in Fourier spaces, and propagates phase mixing and Landau damping in weighted $L^\infty_{k,p}$ norms.

\end{abstract}

\tableofcontents

\section{Introduction}

The Hartree--Fock theory is widely used in plasma physics and quantum mechanics to describe the effective dynamics of many weakly interacting electrons and fermions in the meanfield regime. In this work, we are interested in the relaxation and asymptotic behavior of solutions to such a theory near homogenous equilibria. Specifically, we consider the following time-dependent Hartree--Fock equation 
\begin{equation}
	\label{HFeqs}
	\begin{cases}
		\begin{aligned}
			&i \partial_{t} \Gamma = [ - \Delta + \cP_\Gamma, \, \Gamma ], \\
			&\Gamma_{\vert_{t=0}} = \Gamma_{0},
		\end{aligned}
	\end{cases}
\end{equation}
that describes the dynamics of a one-particle density matrix $\Gamma$, which is a self-adjoint operator on $L^2(\mathbb{R}^d)$, $d\ge 3$, satisfying the Pauli exclusion constraint $0 \le \Gamma \le 1$. Here in \eqref{HFeqs}, $\cP_\Gamma$ is the self-consistent operator defined by 
\begin{equation}\label{defP}
\cP_\Gamma: = w_1 \star_x \rho_{\Gamma} - \mathcal{X}_{w_2,\Gamma}
\end{equation}
which accounts for the classical {\em meanfield interaction} effect through the spatial convolution $w_1 \star_x \rho_{\Gamma}$ and for the quantum effect through the {\em exchanging operator} $\mathcal{X}_{w_2,\Gamma}$, where $\rho_\Gamma$ denotes the configuration space density, and $w_1(x), w_2(x)$ are two interaction potentials. The convolution $w_1 \star_x \rho_{\Gamma}$ is a space function, and thus understood as a multiplication operator. When $\Gamma$ has an integral kernel $\Gamma(x,y)$, the  
the space density and the exchange kernel are computed by  
\begin{equation}\label{defX} 
\rho_\Gamma(x) = \Gamma(x,x), \qquad \mathcal{X}_{w_2,\Gamma}(x,y) = w_2(x-y) \Gamma(x,y).
\end{equation}
Note that both the meanfield effect and the exchange operator are self-consistent, and therefore, the coupling in \eqref{HFeqs} is nonlinear with the usual bracket notation $[A,B]=AB-BA$ for commutators. In practice, the exchange operator is often sub-order with $|w_2|\ll1$, e.g., in the semiclassical regime. In the absence of the exchange operator $\mathcal{X}_{w_2,\Gamma}\equiv 0$, the equation \eqref{HFeqs} is often referred to as the Hartree equation (of infinitely many interacting fermions).
 
 The Cauchy problem for the Hartree and Hartree--Fock equation \eqref{HFeqs} is rather classical, going back to the works \cite{Bove1974, Bove1976, Chadam1976, Chadam1975, Zagatti1992} in the 70s. The stability and scattering theory near vacuum for the Hartree--Fock equations were also established in \cite{Wada2002,Ikeda2012,Maleze2025}, while the modified scattering theory for the Hartree equations in the critical Coulomb case was obtained recently in \cite{NguyenYou2024}. In the Kohn--Sham density functional theory, the non-local exchange operator is replaced by the local space density, and the corresponding Cauchy problem as well as small data scattering theory in almost all subcritical regimes were established in  \cite{Jerome2015,Sprengel2017} and in \cite{Pusateri2021}, respectively. 
 
This paper concerns with the asymptotic stability and scattering theory near homogenous steady states. The Hartree--Fock equations \eqref{HFeqs} indeed admit a large  class of translation-invariant equilibria of the form 
\begin{equation}\label{equilibria}
\gammaeq = f(-\Delta)
\end{equation} 
for arbitary functions $0 \le f \le 1$ with finite space density $\rho_{\gammaeq}= (2\pi)^{-d}\int_{\mathbb{R}^d} f(|k|^2) \, dk$. Apparently, this includes several physically important examples such as Fermi gases and Gibbs states at positive temperature. 

In their seminal works \cite{Lewin2014, Lewin2015}, Lewin and Sabin initiated the well-posedness and scattering theory of the nonlinear Hartree equation around homogenous equilibria of the form \eqref{equilibria}. In particular, they established the asymptotic stability of such equilibria in the case of screened or short-range potentials in $\RR^2$. This has been generalized to the three-dimensional case in \cite{Chen2018}. Further generalizations that allow less regular equilibria and interaction potentials can be found in \cite{Chen2017,Hadama2025b, Hadama2024, Collot2020, Collot2022, You2024}, including the semiclassical regime \cite{Lewin2020, Smith2024, Hadama2025}. In all these works, the interaction potential is assumed to be screened or short range: namely, $w_1 \in L^1$, or more generally, $w_1$ is a Borel measure with $\hw_1 \in L^{\infty}$, which turns out to be necessary for the invertibility of the corresponding linearized problem. For long-range interaction potentials (e.g. the classical Coulomb potential), this invertibility always fails for any non-trivial equilibrium due to resonances or poles of the linearized spacetime symbol, leading to 
the appearance of {\em plasmons} or quantized plasma oscillations \cite{NguyenYou2025}. 

Much fewer works were found for the Hartree--Fock equation \eqref{HFeqs} near equilibria in the presence of the exchange operator, most probably due to the common folklore that the exchange term may often be negligible (e.g. in the semiclassical regime). Interestingly, however, the inclusion of the exchange term, despite being sufficiently small, leads to several complications that attract our attention, which we shall highlight in Section \ref{sec-issues} below. To the best of our knowledge, the first available work that studies the Hartree--Fock equations near equilibria in the presence of a small exchange term was carried out in \cite{Collot2025}, establishing the scattering of homogenous equilibria using random field formulation in higher dimensions $d \ge 4$. There appear no available works that treat the three-dimensional case, which we shall provide in this paper. In particular, we establish nonlinear Landau damping and scattering theory for the Hartree--Fock equations near a large class of homogenous equilibria. 

\subsection{Equilibria and potentials}\label{sec-Assmp}

Specifically, let $w_1,w_2$ be interaction potentials, and consider {\em positive} homogenous equilibria of the form $\gammaeq=f(-\Delta)$ in the whole space $\RR^d$ with $d\ge 3$. Throughout the paper, we assume the followings. 

\begin{enumerate}
	\labitem{\textbf{(H1)}}{equi:nonneg} {$0 < f \le 1$} with finite space density $\rho_{\gammaeq}= (2\pi)^{-d}\int_{\mathbb{R}^d} f(|k|^2) \, dk$. 
	
	\labitem{\textbf{(H2)}}{equi:reg} $f(\cdot)$ belongs to $C^{n_0}(\RR_+)$ for some $n_0 > d+3$, satisfying  
	$$| \partial_e^{n} f(e)| \lesssim \langle e \rangle^{-n_1 - n}$$ 
	for all $e \ge 0$ and $0 \le n \le n_0$, for some $n_1 > d$. 
	
	\labitem{\textbf{(H3)}}{assumption:potential} The potentials $w_j$, $j=1,2$, are radial Borel measures on $\mathbb{R}^d$ such that $\hw_j$ are continuous, nonnegative, and bounded functions.
	
	\labitem{\textbf{(H4)}}{assumption:exchange} $\hw_2 \in W^{2n_0,\infty}$ with a sufficiently small norm $\epsilon_1: = \|\hw_2\|_{W^{2n_0,\infty}}>0$. 

\end{enumerate}

Assumptions {\ref{equi:nonneg}-\ref{equi:reg}} are standard assumptions that are satisfied by a large class of equilibria where the regularity and decay assumptions are linked to the regularity and mixing rates of the Green function to the linearized problem. 
In addition, we consider precisely the positive equilibria and short-range potentials (i.e. $\hw_1(0)<\infty$), for otherwise the linear stability may not hold, among other possible issues, see, e.g., \cite{Lewin2014, Lewin2015, HanKwan2021, HanKwan2021a,HKNR5, Nguyen2026, NguyenYou2025, You2024}. Finally, we do not make any smallness assumption on the equilibria, but the smallness on the exchange potential $w_2$. As will be clear below, see Section \ref{sec-issues}, the smallness plays a crucial role in retaining the Schr\"odinger type dispersion and ensuring the linear stability.

\subsection{Main results}

Our main results concern the large time behavior of solutions to the Hartree--Fock equations \eqref{HFeqs} near equilibria $\gamma_f$ under small initial perturbations. Precisely, we establish the following main theorem.  

\begin{theorem}\label{thm:main} 
Fix an $n_0 >d+3$, $n_1>d$, $N_0 \ge n_0 -3$, and $\sigma_0 >2d + 1$. Let $w_1, w_2$ be the interaction potentials and $\gammaeq=f(-\Delta)$ be an equilibrium of \eqref{HFeqs} as described in Section \ref{sec-Assmp}. Let $\Gamma_0 = \gammaeq+\gamma_0$ be the initial density operator satisfying
	\begin{equation}\label{assm:initial}
		\epsilon:=\sum_{|\alpha| \le N_0}  \left\| \langle x \rangle^{\lfloor d/2 \rfloor +1} \langle \nabla \rangle^{\sigma_0} \operatorname{ad}_x^{\alpha}(\gamma_0) \langle \nabla \rangle^{\sigma_0} \langle x \rangle^{\lfloor d/2 \rfloor +1}  \right\|_{\textnormal{HS}}
	\end{equation}
	where $\operatorname{ad}_x(A) = [x,A]$ and $\|A\|_{\textnormal{HS}} = \|A\|_{L^2_{x,y}}$.
	Then, there exists an $\epsilon_0 >0$ so that, provided $\epsilon \in [0,\epsilon_0]$, there is a global-in-time solution $\Gamma(t)$ to \eqref{HFeqs}, and its associated space density $\rho_{\Gamma}(t)$ satisfies
	\[
	\| \partial_x^{n} (\rho_{\Gamma}-\rho_{\gamma_f})(t) \|_{L^{p}} \lesssim \epsilon_0 \langle t \rangle^{-d(1-1/p)-n+\delta}
	\]
for $2\le p\le \infty$, $0\le n < \min\{n_0-3, \sigma_0-1\}-d$, and for small $\delta >0$. In addition, there is a unique operator $\gamma_{\infty} \in \mathrm{HS}$ satisfying
	\[
		\| \Gamma(t) -\gammaeq -e^{itH_\infty}  \gamma_{\infty} e^{-itH_\infty} \|_{\textnormal{HS}} \lesssim \epsilon_0 \langle t \rangle^{-d/2+\delta}
	\]
with the limiting Hamiltonian $H_\infty= -\Delta - \mathcal{X}_{w_2, \gammaeq}
$.  
\end{theorem}

Theorem \ref{thm:main} establishes the asymptotic stability and large time behavior of solutions near a large class of homogenous equilibria to the Hartree--Fock equations, which appears to be the first such a result in the physical space dimension $d= 3$. In addition, the configuration space density disperses in space with rate of order $t^{-d+\delta}$ in addition to {\em phase mixing} estimates for spatial derivatives. As a result, we also establish the scattering theory of the Hartree--Fock solutions which scatter to solutions of a linear dynamics $e^{itH_\infty}  \gamma_{\infty} e^{-itH_\infty}$ in the large time limit. We note that the required localization and regularity assumptions on the initial perturbations may not be optimal. We also remark that although the theorem is stated under the stronger assumption \eqref{assm:initial}, the proof only uses the weaker norm
\[
	\sum_{|\alpha| \le N_0} \left\| \langle k \rangle^{\sigma_0} \langle p \rangle^{\sigma_0} (\partial_k - \partial_p)^{\alpha} \widehat{\gamma}_0 \right\|_{L^{\infty}_{k,p}} \lesssim \epsilon 
\]
on the Fourier side for $N_0>d$ and $\sigma_0>2d+1$.

\subsection{Difficulties and main ideas}\label{sec-issues}

We shall now discuss the difficulties and main ideas to overcome them in establishing the main results stated in Theorem \ref{thm:main}. The very first issue is to analyze the dispersive nature of the expected limiting dynamics $e^{itH_\infty}  \gamma_{\infty} e^{-itH_\infty}$, with Hamiltonian $H_\infty = -\Delta - \mathcal{X}_{w_2, \gammaeq}$, whose dispersion relation reads 
\begin{equation}\label{w}\omega(k) = |k|^2 - \FcX_{w_2, \gammaeq}(k),
\end{equation}
for a sufficiently smooth Fourier multiplier $\FcX_{w_2, \gammaeq}(k)$, see \eqref{def:omega}. For a small potential $w_2$, $\omega(k)$ clearly behaves like a standard Schr\"odinger dispersion. However, the main relaxation mechanism of the Hartree--Fock space density is the phase mixing dictated by the commutator $[H_\infty, \gamma]$, which is of a transport type dynamics (e.g., through the Wigner transform), and a small perturbation of the dispersion relation \eqref{w} could lead to substantial obstructions in utilizing phase mixing. 
Specifically, the space density $\rho_\gamma$ of the Hartree--Fock dynamics $e^{itH_\infty}  \gamma_0 e^{-itH_\infty}$ can be computed through its Fourier transform   
\begin{equation}\label{Fr0}
\Frho_\gamma(t,k) = \int e^{-itA_{k-p,p}} \Fgamma_0(k-p,p) \; dp
\end{equation}
with the phase function $A_{k,p} = \omega(k) - \omega(p)$ (i.e. the Fourier symbol of the commutator $[H_\infty, \gamma]$). The relaxation of the space density is thus due to momentum averaging and time oscillation $e^{-itA_{k-p,p}} $, a classical relaxation mechanism, known as {\em phase mixing}. As $\omega(k)\sim |k|^2$, the phase $A_{k-p,p}$ remains of a transport type dispersion, namely $A_{k-p,p} \sim k\cdot p$, leading to a rapid decay of the space density in $|kt|$. However, several issues arise.  

First, we will need to study the linear response from the background electrons $\gammaeq = f(-\Delta)$, namely the commutator term $[\cP_\gamma, \gammaeq]$. This involves the classical meanfield effect and the quantum exchange term. For the meanfield effect, we do not assume any smallness on the interaction potential $w_1(x)$. As a result, we will need to invert the classical Penrose dielectric function $D(\lambda,k)$, see \eqref{def:disp}. The fact that the transport dispersion $A_{k-p,p}$ is no longer linear in $p$ requires careful modifications from the earlier developed Fourier--Laplace framework, see \cite{NguyenYou2025, You2024}. For quantum exchange effect, the linear response now involves  the full range of Fourier frequencies, and can no longer be solved mode by mode independently as was the case in the classical Vlasov and Hartree theory, see, e.g., \cite{Mouhot2011,  Grenier2021, Bedrossian2022, HanKwan2021a,Ionescu2023, Nguyen2026, Lewin2014, Lewin2015, NguyenYou2025}. 

Second, the modification in the transport dispersion $A_{k-p,p} \sim k\cdot p$ leads to delicate quantum echo resonances. Indeed, the nonlinear interaction $[\cP_\gamma, \gamma]$ involves particles traveling forward $e^{-itA_{k-p,p}}$ with group velocity $\partial_pA_{k-p,p}$ and those traveling backward $e^{isA_{p,\ell-p}}$ with group velocity $\partial_pA_{p,\ell-p}$, 
leading to resonances at $t\partial_pA_{k-p,p} \sim s\partial_pA_{p,\ell-p}$, resembling the classical plasma echoes at $kt \sim \ell s$. Due to the presence of the exchange term, the resonant set depends non-trivially on $p\in \RR^d$, and in fact has $p$-derivatives grow linearly in $t$, which in turn leads to a derivative loss (cf. the $p$-independent resonant condition $kt \sim \ell s$ in the classical case). To overcome this issue, we develop a nonlinear iterative scheme that utilizes the transport type dispersion in Fourier spaces and propagates phase mixing in weighted $L^\infty_{k,p}$ norms. The treatment of such complex momentum-dependent echo resonances is one of the novelties in this work. For a related nonlinear scheme in Fourier spaces, see \cite{Bedrossian2018, Nguyen2020, You2024, Smith2024} for the classical Vlasov and Hartree theory where echo resonances are momentum-independent (i.e. $kt\sim \ell s$).   

Finally, we remark that the aforementioned issues are also the central obstructions in the classical Vlasov and Hartree theory near {\em spatially inhomogenous} equilibria, which are far from being understood. For a study of the linearized dynamics, see \cite{Guo, Despres, Hadzic}.

\subsection{Notation}

Throughout the paper, we shall deal with functions of the form $A(k-p,p)$. When no confusion is possible, we write $\partial_p A(k-p,p) = \partial_p [A(k-p,p)].$ We also use repeatedly the inequality 
\begin{equation}\label{minab}
\langle a\rangle^{-\sigma} \langle b\rangle^{-\sigma} \le 2^\sigma \langle a-b\rangle^{-\sigma} \min \{ \langle a \rangle, \langle b\rangle\}^{-\sigma},
\end{equation}
for any vectors $a,b$, and for $\sigma>0$, which is direct by considering two cases: $|a|\ge |a-b|/2$ or $|b|\ge |a-b|/2$.
We use the notation $\widehat{~\cdot~}$ and $\widetilde{~\cdot~}$ to denote the Fourier transform in $\RR^d_x$ and the Laplace--Fourier transform in $\RR_+ \times \RR^d_x$, namely   
\[
\widehat{f}(k) = \int_{\mathbb{R}^{d}} e^{-ix \cdot k} f(x) \, dx, 
\qquad	\widetilde{f}(\lambda,k) = \int_{0}^{\infty} e^{-\lambda t} \widehat{f}(t,k) \, dt ,
\]
for $\Re \lambda \ge 0$ and $k \in \RR^d$. For integral kernels $K(x,y)$, we denote by $\FK(k,p)$ their Fourier transform in variables $x,y\in \RR^d$ with dual variables $k,p$ in $\RR^d$, respectively.

\section{Preliminaries}

\subsection{Perturbations}

Throughout the paper, we shall work with Hartree--Fock solutions in a perturbative form 
\begin{equation}\label{pertGamma}
\Gamma =  \gammaeq + \gamma
\end{equation}
near equilibria $\gammaeq = f(-\Delta)$. Using \eqref{HFeqs}, this leads to the following perturbed Hartree--Fock equation
\begin{equation}\label{eqn:perturb}
	i\partial_t \gamma = [H_\infty, \, \gamma] + [\cP_\gamma, \, \gammaeq] + [\cP_\gamma, \, \gamma],
\end{equation}
in which we recall that $H_\infty = -\Delta - \mathcal{X}_{w_2, \gammaeq}$ and $\cP_\gamma = w_1 \star_x \rho_{\gamma} - \mathcal{X}_{w_2, \gamma}$, where space density $\rho_\gamma$ and exchange operator $\cX_{w_2, \gamma}$ are defined as in \eqref{defX}. In particular, we note that $ w_1 \star_x \rho_{\gamma} $ is a constant at the equilibrium $\gamma_f$, and therefore does not contribute to the limiting Hamiltonian $H_\infty$. We also observe that the first two terms on the right hand side of \eqref{eqn:perturb} are linear in $\gamma$, while the last term is nonlinear, since $\cP_\gamma$ is self-consistently generated by $\gamma$. 

We shall solve the Hartree--Fock equation \eqref{eqn:perturb} with 
small initial data $\gamma_{\vert_{t=0}} = \gamma_0$ as described in Theorem \ref{thm:main}. In particular, we will only work with Hilbert--Schmidt operators $\gamma(t)$ with integral kernels $\gamma(t,x,y)$. We note that the commutator $[A,B]$ is computed through 
\begin{equation}\label{commAB}
[A,B](x,y) = \int \Big( A(x,z)B(z,y) - B(x,z)A(z,y)\Big) \; dz,
\end{equation}
where $A(x,y)$ and $B(x,y)$ are integral kernels of $A, B$, respectively.

\subsection{Fourier transform}

We will work in Fourier spaces. Denote by $\Fgamma(t,k,p)$ the Fourier transform of the integral kernel $\gamma(t,x,y)$ in variables $x,y$ with dual variables $k,p$ in $\RR^d$, respectively. We obtain the following simple lemma. 

\begin{lemma}\label{lem-FHF} Let $g(k) = f(|k|^2)$ be the equilibrium, and introduce $a_{k,p}=g(k) - g(p)$ and $A_{k,p} = \omega(k) -\omega(p)$, with the modified Schr\"odinger dispersion relation 
\begin{equation}\label{def:omega}
	\omega(k): = |k|^2 - \frac{1}{(2\pi)^d} (\widehat{w}_2 \star g)(k) +  \frac{1}{(2\pi)^d} (\widehat{w}_2 \star g)(0).
\end{equation}
Then, for any Hilbert--Schmidt operator $\gamma(t)$ solving 
\eqref{eqn:perturb}, its Fourier transform $\Fgamma(t)$ satisfies
\begin{equation}\label{eqn:fourier}
	\begin{aligned}
		i\partial_t \hgamma(t,k,p) 
		&= A_{k,p} \hgamma(t,k,p) - a_{k,p}\FcP_\gamma(t,k,p)\\
		&\quad 
		+ \frac{1}{(2\pi)^d} \int \left(\FcP_{\gamma}(t,k,\ell) \widehat{\gamma}(t,-\ell, p) -  \widehat{\gamma}(t,k,\ell) \FcP_{\gamma}(t,-\ell, p) \right)  d\ell,
	\end{aligned}
\end{equation}
for any $k,p\in \RR^d$, where $\FcP_\gamma(t,k,p)$ is computed by   
\begin{equation}\label{comFcP}
\FcP_\gamma(t,k,p) = \hw(k+p) \widehat{\rho}_{\gamma}(t,k+p) - \widehat{\mathcal{X}}_{\gamma}(t,k,p).
\end{equation}
\end{lemma}
\begin{proof} The lemma follows from taking the Fourier transform of the Hartree--Fock equation \eqref{eqn:perturb}. Indeed, we first note that the Fourier transform of the commutator $[A,B]$, see \eqref{commAB}, is computed by  
\begin{equation}\label{FcommAB}
\mathcal{F}_{x,y}([A,B]) (k,p) = \frac{1}{(2\pi)^d} \int \Big( \FA(k,\ell)\FB(-\ell,p) - \FB(k,\ell)\FA(-\ell,p)\Big) \; d\ell,
\end{equation}
where $\FA(k,p), \FB(k,p)$ are the Fourier transform of $A(x,y), B(x,y)$, respectively. This yields the last nonlinear term appearing in \eqref{eqn:fourier}. As for the linear term, we first note that the equilibrium operator $\gamma_f = g(-i\nabla_x)$ has its Fourier transform 
\begin{equation}\label{Fequilibria}\Fgammaeq(k,p) = (2\pi)^{-d} g(k) \delta(k+p).\end{equation}
In addition, the Fourier transform of the space density and the exchange operator is computed by 
\begin{equation}\label{Frho}
\begin{aligned}
\Frho_\gamma(t,k) = \int \Fgamma(t,k-\ell,\ell) \; d\ell,
\qquad \widehat{\cX}_\gamma(t,k,p) = \int \hw_2(\ell) \widehat{\gamma}(t,k-\ell, p+\ell) \, d\ell. 
\end{aligned}
\end{equation}
Using \eqref{FcommAB}, we obtain the lemma at once. 
\end{proof}

\begin{lemma}\label{lem-FHFmu}
For any Hilbert--Schmidt operator $\gamma(t)$ solving 
\eqref{eqn:perturb}, the conjugate operator 
\begin{equation}\label{defmu}
\mu(t) = e^{-itH_\infty} \gamma(t) e^{itH_\infty}
\end{equation}
has its Fourier transform $\Fmu(t)$ satisfying
\begin{equation}\label{eqn:profile}
	\begin{aligned}
		i\partial_t \widehat{\mu}(t,k,p) 
		&= -a_{k,p} e^{itA_{k,p}} \widehat{\cP}_{\gamma} (t,k,p)
+ (2\pi)^{-d} \int \Big( e^{itA_{k, \ell-k}} \widehat{\cP}_{\gamma}(t,k, \ell-k) \widehat{\mu}(t,k-\ell, p) 
\\&\qquad - e^{-itA_{p,\ell-p}} \widehat{\mu}(t,k,p-\ell) \widehat{\cP}_{\gamma}(t,\ell-p, p) \Big) d\ell,
	\end{aligned}
\end{equation}
where $a_{k,p}$ and $A_{k,p}$ are defined as in Lemma \ref{lem-FHFmu}. 
\end{lemma}
\begin{proof} The lemma is direct, upon noting that $\widehat{\mu}(t,k,p) = e^{itA_{k,p}}\widehat{\gamma}(t,k,p)$ and using the fact that 
$A_{k,p} - A_{\ell,p} = A_{k,\ell}$ and $A_{k,p} - A_{k,\ell} = -A_{p,\ell}$. \end{proof}

\begin{remark}\label{rem-FcP} It turns out to be crucial to note that the meanfield and exchange kernel $\cP_\gamma$ that appears in the profile equation \eqref{eqn:profile} is consistently of the form $\widehat{\cP}_{\gamma}(t,k-p, p)$, for which, recalling \eqref{comFcP}, we have 
\begin{equation}\label{comFcPinv}
\FcP_\gamma(t,k-p,p) = \hw(k) \widehat{\rho}_{\gamma}(t,k) - \widehat{\mathcal{X}}_{\gamma}(t,k-p,p).
\end{equation}
Next, in view of \eqref{Frho}, we compute
\begin{equation}\label{compcXkp}
\widehat{\cX}_\gamma(t,k-p,p) = \int \hw_2(\ell-p) \widehat{\gamma}(t,k-\ell, \ell) \, d\ell.
\end{equation}
Namely, in the appropriate coordinate shift, the exchange operator $\cX_\gamma(t)$ plays a similar role as does the space density $\rho_\gamma(t)$, recalling again \eqref{Frho}. Namely, we may expect that the exchange operator enjoys the same dispersion and phase mixing estimates as those of the space density!  
\end{remark}

\subsection{Elementary lemmas}

In this section, we record a few elementary lemmas that will be used throughout the paper. 

\begin{lemma}\label{lem:Akp}
Let $A_{k,p} = \omega(k) -\omega(p)$ with $\omega(k)$ defined as in \eqref{def:omega}. Then, there are universal positive constants $\theta_0, \theta_1, C_0$ 
so that 
\begin{equation}\label{dispAkp}
\theta_0 |k| \le |\partial_p A_{k-p,p}| \le \theta_1|k|, \qquad |  \partial_p^{\alpha}A_{k-p,p}| \le C_0 |k|
\end{equation} 
uniformly in $k$ and $p$, for any multi-index $2\le |\alpha| \le 3n_0 -1$. 
\end{lemma}
\begin{proof} We write $\omega(k) = |k|^2 - m(k)$, where
	\[
		m(k):= - \frac{1}{(2\pi)^d} (\hw_2 \star g)(k) + \frac{1}{(2\pi)^d} (\hw_2 \star g)(0).
	\]
Recall that $g \in W^{n_0, 1}$ and $\hw_2 \in W^{2n_0,\infty}$ with a sufficiently small norm, see Assumption \ref{assumption:exchange}. 
As a result, $\| \nabla^2 m \|_{L^{\infty}} \ll1$. Therefore, using the mean value theorem and the fact that $\hw_2$ and $g$ are radial, we bound 
	\[
		|\partial_p A_{k-p,p}| = |-2k - \nabla m(k-p) + \nabla m(-p)| \ge (2- \| \nabla^2 m \|_{L^{\infty}}) |k| \ge \theta_0|k|
	\]
for some $\theta_0>0$. Similarly, we have
	\[
		| \partial_p A_{k-p,p} | \le (2 +\| \nabla^2 m \|_{L^{\infty}}) |k|,
	\]
proving that $|\partial_p A_{k-p,p}| \sim |k|$. Higher derivatives follow similarly.  
\end{proof}

\begin{lemma}\label{lem:akp} Let $a_{k,p}=g(k) - g(p)$ with $g(k)= f(|k|^2)$ being the equilibrium. Let $n_0, n_1$ be the regularity and decay indices as in Assumption \ref{equi:reg}. Then, for all $|\alpha| \le n_0 -1$, we have
\begin{equation}\label{dervakp}
		|\partial_p^{\alpha} a_{k-p,p}| \lesssim \frac{|k|}{\langle k \rangle} \left( \langle k-p \rangle^{-2n_1} +  \langle p \rangle^{-2n_1} \right)
	\end{equation}
	uniformly in $k$ and $p$. 
\end{lemma}
\begin{proof}
	Let $|\alpha| \le n_0 -1$. By Assumption \ref{equi:reg}, $|\partial_p^{\beta} g(p)| \lesssim \langle p \rangle^{-2n_1 - |\beta|}$ for $|\beta| \le n_0$. When $|k| \le 1$, we bound 
	\[
		\partial_p^{\alpha}a_{k-p,p}  = (\partial_p^{\alpha} g)(p-k) - (\partial_p^{\alpha}g)(p)
		\lesssim |k| \sup_{t \in [0,1]} |\partial^{|\alpha| +1} g (p-tk)| \lesssim |k| \langle p \rangle^{-2n_1}
	\]
	upon using the mean value theorem.
	When $|k| \ge 1$, we simply bound
	\[
		|\partial_p^{\alpha}a_{k-p,p}| = |\partial_p^{\alpha} g(k-p) | + |\partial_p^{\alpha} g(p)|
		\lesssim \langle p-k \rangle^{-2n_1} + \langle p \rangle^{-2n_1},
	\]
	which follows from the decay of $\partial^{\alpha}_p g(p)$. 
\end{proof}

\begin{lemma}\label{lemAphase}
Let $A_{k,p} = \omega(k) -\omega(p)$ with $\omega(k)$ defined as in \eqref{def:omega}. Then, for any $F \in W^{N,1}$ with $N \le 3n_0 -2$, there holds
	\[
		\left| \int e^{-it A_{k-p,p}} F(p) \, dp \right| \lesssim \langle k t\rangle^{-N} \| F \|_{W^{N,1}},
	\]
	uniformly in $k\not =0$. 
\end{lemma}
\begin{proof}
The lemma is a direct application of Lemma \ref{lem:IBP}, upon using \eqref{dispAkp} and noting that the ratio of $\partial_p^\alpha A_{k-p,p}$ to $\partial_p A_{k-p,p}$ is uniformly bounded in $k\not =0$ for any $|\alpha|\ge 2$.  
\end{proof}

\section{Linear analysis}

In this section, we study the linear inhomogeneous equation
\begin{equation}\label{eqn:perturb-linear}
	i\partial_t \gamma = [-\Delta - \mathcal{X}_{w_2, \gammaeq}, \, \gamma] + [w_1 \star_x \rho_{\gamma}, \, \gammaeq] + F
\end{equation}
with initial data $\gamma(0) = \gamma_0$, where $F=F(t)$ is an arbitrary forcing term. We follow the Fourier--Laplace framework \cite{NguyenYou2025, You2024} that were developed earlier to treat the linearized Hartree equations (i.e. linearization of \eqref{eqn:perturb} without the exchange term). From now on, we drop the subscripts when no confusion is possible.

\subsection{Penrose--Lindhard dispersion}
We begin with introducing the Penrose--Lindhard dispersion relation $D(\lambda, k)$.

\begin{proposition}\label{prop:resolvent-linear}
	Consider the linear equation \eqref{eqn:perturb-linear} with initial data $\gamma(0)= \gamma_0$. Let $\rho(t,x) = \gamma(t,x,x)$ be the density associated with the solution of $\eqref{eqn:perturb-linear}$. Let $\widetilde{\rho}(\lambda, k)$ denote the Fourier--Laplace transform of $\rho$. For $\Re \lambda >0$ and $k \in \mathbb{R}^d$, one can write
	\begin{equation}\label{eqn:resolvent-linear}
		D(\lambda, k) \widetilde{\rho}(\lambda, k) 
		= \widetilde{S}(\lambda,k)
	\end{equation}
	where
	\begin{equation}\label{def:disp}
		D(\lambda, k) = 1 - i\hw_1(k)\int_{\RR^d} \frac{a_{k-p,p}}{\lambda + i A_{k-p,p}} \, dp,
	\end{equation}
	and
	\begin{equation}\label{def:cS}
	\widetilde{S}(\lambda,k)= \int \frac{\widehat{\gamma}_0 (k-p,p)}{\lambda + i A_{k-p,p}} \, dp - \int \frac{i\widetilde{F}(\lambda, k-p,p)}{\lambda + i A_{k-p,p}} \, dp.
	\end{equation}
\end{proposition}
\begin{proof}
	We first write \eqref{eqn:perturb-linear} in terms of integral kernels $\gamma(x,y)$ and take the Fourier transform in $x$ and $y$, with dual variables $k$ and $p$. Then we obtain
	\begin{equation}\label{eq:HF-Fourier2}
		\begin{aligned}
			i\partial_t \hgamma(t,k,p) 
			&= A_{k,p} \hgamma(t,k,p) - a_{k,p}\hw_1(k+p) \widehat{\rho}(t,k+p) + \widehat{F}(t,k,p).
		\end{aligned}
	\end{equation}
	Next, we evaluate \eqref{eq:HF-Fourier2} at $(k-p,p)$ and take the Laplace transform in $t$ with the dual variable $\lambda$. This yields
	\[
	(\lambda + i A_{k-p, p}) \widetilde{\gamma}(\lambda, k-p,p) = \widehat{\gamma}_0 (k-p,p) + i a_{k-p,p} \hw_1(k) \widetilde{\rho}(\lambda, k) -i  \widetilde{F}(\lambda, k-p,p).
	\]
	Dividing by $\lambda + i A_{k-p, p}$, we get 
	\[
	\widetilde{\gamma}(\lambda, k-p, p) = \frac{\widehat{\gamma}_{0}(k-p,p)}{\lambda + i A_{k-p,p}} + \frac{i a_{k-p,p}}{\lambda + iA_{k-p,p}} \hw_1(k) \widetilde{\rho}(\lambda,k) - \frac{i\widetilde{F}(\lambda, k-p,p)}{\lambda + i A_{k-p,p}} .
	\]
	Integrating with respect to $p$ yields
	\[
	\widetilde{\rho}(\lambda, k) = \int \frac{\widehat{\gamma}_0 (k-p,p)}{\lambda + i A_{k-p,p}} \, dp + \hw_1(k) \widetilde{\rho}(\lambda,k) \int \frac{ia_{k-p,p}}{\lambda + iA_{k-p,p}} \, dp - \int \frac{i\widetilde{F}(\lambda, k-p,p)}{\lambda + i A_{k-p,p}} \, dp.
	\]
	Rearranging this gives 
	\[
		D(\lambda, k)\widetilde{\rho}(\lambda, k) = \widetilde{S}(\lambda, k),
	\]
	where $D(\lambda, k)$ and $\widetilde{S}(\lambda, k)$ are defined by \eqref{def:disp} and \eqref{def:cS}, respectively. This completes the proof of \Cref{prop:resolvent-linear}.
\end{proof}

The following proposition extends $D(\lambda,k)$ to the boundary $\{ \Re \lambda =0\}$ for each $k \neq 0$.
\begin{proposition}\label{prop:disp-extension}
	For each $k \in \mathbb{R}^d \setminus \{0 \}$, the dispersion relation $D(\lambda,k)$ admits a continuous extension to $\{ \Re \lambda \ge 0\}$. It satisfies 
	\begin{equation}\label{eqn:disp-timeint}
		D(\lambda, k) = 1 -i \hw_1(k) \int_0^{\infty} e^{-\lambda t} \int_{\RR^d} e^{-itA_{k-p,p}} a_{k-p,p} \, dp \, dt
	\end{equation}
	for $\Re \lambda \ge 0 $ and $k \neq 0$.
\end{proposition}

\begin{proof}
Fix $k \neq 0$. The expression \eqref{eqn:disp-timeint} for $\Re \lambda >0$ follows from the definition \eqref{def:disp}. We set
\[
	J(t,k)= \int_{\RR^d} e^{-itA_{k-p,p}} a_{k-p,p} \, dp .
\]
Integrating by parts in $p$ and using \Cref{lem:Akp,lem:akp}, we get
\[
	\int_0^{\infty} |J(t,k)| \, dt
	\lesssim  \| a_{k-p,p} \|_{W^{2,1}} \int_0^{\infty} \langle kt \rangle^{-2} \, dt  \lesssim  \langle k \rangle^{-1}.
\]
Therefore, for each $k \neq 0$, 
\[
	D(\lambda, k) = 1- i\hw_1(k) \int_0^{\infty} e^{-\lambda t} J(t,k) \, dt
\]
defines a continuous extension to $\{ \Re \lambda \ge 0 \}$.
\end{proof}

\subsection{Penrose--Lindhard stability}

The main goal of this section is to prove that the equilibria $\gammaeq = f(-\Delta)$ described as in Section \ref{sec-Assmp} are spectrally stable and satisfy the strong  
Penrose--Lindhard stability condition. Precisely, we obtain the following.

\begin{theorem}\label{thm:Penrose-Lindhard}
	Let $\gammaeq= f(-\Delta)$ be the equilibrium, and let $w_j$, $j=1,2$, be the interaction potentials. Assume that $f$ and $w_j$ satisfy \ref{equi:nonneg}--\ref{assumption:exchange}. Then
	\begin{equation}\label{eq:Penrose-Lindhard}
		\inf_{k \in \mathbb{R}^d \setminus \{ 0\}} \inf_{\Re \lambda \ge 0} | D(\lambda, k)| \ge c_0
	\end{equation}
	for some $c_0 >0$.
\end{theorem}

The condition \eqref{eq:Penrose-Lindhard} is often referred to as the strong  
Penrose--Lindhard stability condition, a quantum analogue of the classical Penrose stability condition in plasma physics. \Cref{thm:Penrose-Lindhard} implies invertibility of the linearized operator (without the exchange term) in $L^2_{t,x}$ whose spacetime Fourier--Laplace symbol is $D(\lambda,k)$. Recall that \Cref{prop:disp-extension} gives an extension of $D(\lambda, k)$ to $\{ \Re \lambda =0 \}$ for each $k \neq 0$. We remark that the behavior of $D(\lambda, k)$ near $(0,0)$ is delicate. For small $|k|$, both $a_{k-p,p}$ and $A_{k-p,p}$ are of order $|k|$, so that $D(\lambda, k)$ might have a nontrivial limit when $|\lambda| \sim |k|$, see \eqref{def:disp}. In view of this, we define
\begin{equation}\label{def:tilde-aA}
	\widetilde{a}_{k,p}:= \frac{a_{k-p,p}}{|k|}, \qquad \widetilde{A}_{k,p} := \frac{A_{k-p,p}}{|k|}.
\end{equation}
Then we obtain that
\begin{equation}\label{eqn:akpAkp-smallk}
 \widetilde{a}_{k,p} \rightarrow -\nabla g(p) \cdot \frac{k}{|k|}, \qquad  \widetilde{A}_{k,p} \rightarrow -\nabla \omega(p) \cdot \frac{k}{|k|}
\end{equation}
as $|k| \to 0$. We therefore introduce the rescaled variable $\tlambda = \lambda/|k|$ and define
\begin{equation}\label{def:Tdisp}
	\cD(\tlambda, k) := D(\tlambda|k|, k) =  1 -i \hw_1(k) \int_0^{\infty} e^{-\tlambda t} \int_{\RR^d} e^{-it \widetilde{A}_{k,p}} \widetilde{a}_{k,p} \, dp \, dt
\end{equation}
for $\Re \tlambda \ge 0$ and $k \neq 0 $, where we used \Cref{prop:disp-extension} and the change of variables $t \mapsto |k|t$.
The next lemma shows that $\cD(\tlambda, k)$ admits a continuous extension to $\{ \Re \tlambda \ge 0 \} \times \mathbb{R}^d$.

\begin{lemma}\label{lem:Tdisp-extension}
	Let $\gammaeq= f(-\Delta)$ be the equilibrium, and let $w_j$, $j=1,2$, be the interaction potentials. Assume that $f$ and $w_j$ satisfy \ref{equi:nonneg}--\ref{assumption:exchange}.
	There exists a continuous extension of $\cD(\tlambda, k)$ to $\{ \Re \tlambda \ge 0\} \times \mathbb{R}^d$, with
	\begin{align*}
		\cD(\tlambda, 0) 
		= 1 + 2 i \hw_1(0) \int_0^{\infty} e^{-\tlambda t} \int_{\RR^d} e^{2it \Omega'(|p|^2) e \cdot p}  f'(|p|^2) e \cdot p \, dp \, dt,
	\end{align*}
	for $\Re \tlambda \ge 0$ and $e \in \mathbb{S}^{d-1}$, where $\Omega(|p|^2) = \omega(p)$. The above definition is independent of $e$, and can be written as
	\[
	\cD(\tlambda, 0)
	= 1 + 2i \hw_1 (0) |\mathbb{S}^{d-2}| \int_0^{\infty} e^{-\tlambda t} \int_0^{\infty} r^d f'(r^2) \int_{-1}^{1} e^{2it\Omega'(r^2) rs} s(1-s^2)^{\frac{d-3}{2}} \, ds \, dr \, dt.
	\]
\end{lemma}

\begin{proof}
	For $k \neq 0$, we define
	\begin{equation}\label{def:Tdisp-J}
		\cJ(t,k) = \int_{\RR^d} e^{-it \widetilde{A}_{k,p}} \widetilde{a}_{k,p} \, dp,
	\end{equation}
	where $\widetilde{A}_{k,p}$ and $\widetilde{a}_{k,p}$ are defined in \eqref{def:tilde-aA}. We observe from \eqref{def:Tdisp} that
	\begin{equation}\label{def:Tdisp-timeint}
		\cD(\tlambda, k) = 1 -i \hw_1(k) \int_0^{\infty} e^{-\tlambda t} \cJ(t,k) \, dt
	\end{equation}
	for $\Re \tlambda \ge 0$ and $k \neq 0$. 
	Write $k= re$ with $r=|k| >0$ and $e= k/|k| \in \mathbb{S}^{d-1}$, and set
	\[
		\cJ(t,r,e):= \cJ(t,re), \qquad \cD(\tlambda, r,e) := \cD(\tlambda, re).
	\]
	We claim that $\cD(\tlambda, r, e)$ extends continuously to $r=0$.
	
	By \eqref{eqn:akpAkp-smallk}, for each $t \ge 0$ and $p \in \RR^d$, we have
	\[
		\widetilde{A}_{re, p} \to -\nabla \omega(p) \cdot e, \qquad \widetilde{a}_{re, p} \to -\nabla g(p) \cdot e
	\]
	as $r \to 0$. Since $\widetilde{a}_{re, p}$ is rapidly decaying in $\min \{ \langle k-p \rangle, \langle p \rangle\}$, we obtain that
	\[
		\cJ(t,r,e) \to \cJ(t,0,e):= - \int_{\RR^d} e^{it\nabla \omega(p) \cdot e} \nabla g(p) \cdot e \, dp
	\]
	for each $t \ge 0$.
	
	Moreover, integrating by parts in $p$ and using \Cref{lem:Akp,lem:akp}, we obtain
	\[
		|\cJ(t,r,e)| \lesssim \langle k \rangle^{-1} \langle t \rangle^{-2}
	\]
	uniformly in $r \ge 0$ and $e \in \mathbb{S}^{d-1}$. Hence,
	\[
		\cD(\tlambda, r,e) \to \cD(\tlambda, 0, e):= 1- i\hw_1(0) \int_0^{\infty} e^{-\tlambda t} \cJ(t,0,e) \, dt
	\]
	as $r \to 0$, uniformly for $\Re \tlambda \ge 0$.
	Using $\nabla \omega(p) = 2 \Omega'(|p|^2)p$ and $\nabla g(p) = 2f'(|p|^2)p$, we obtain that
	\[
		\cD(\tlambda, 0, e)
		= 1 + 2i \hw_1 (0) \int_{0}^{\infty} e^{-\tlambda t}
			\int_{\RR^d} e^{2it\Omega'(|p|^2)e \cdot p} f'(|p|^2) e \cdot p \, dp \, dt.
	\]
	By rotational symmetry, the right hand side is independent of $e$. Hence, we may define $\cD(\tlambda, 0)$ by this expression. This gives an extension of $\cD(\tlambda ,k)$ to $\{ \Re \tlambda \ge 0 \} \times \mathbb{R}^d$.
	
	To derive the remaining formula, we choose $e=(1,0,\cdots, 0)$. Then
	\begin{align*}
		\cD(\tlambda, 0) 
		&= 1 + 2i \hw_1 (0) \int_0^{\infty} e^{-\tlambda t} \int_{\RR^d} e^{2it \Omega'(|p|^2) p_1} f'(|p|^2) p_1 \, dp \, dt \\
		&= 1 + 2i \hw_1 (0) \int_0^{\infty} e^{-\tlambda t} \int_{0}^{\infty} \int_{\mathbb{S}^{d-1}} e^{2it\Omega'(r^2) r\sigma_1} f'(r^2) r^d \sigma_1 \, d\sigma \, dr \, dt .
	\end{align*}
	Consider $\sigma = (s, \sqrt{1-s^2} \phi)$ where $s = \sigma_1 \in [-1,1]$ and $\phi \in \mathbb{S}^{d-2}$. Then we have
	\[
		d\sigma =(1-s^2)^{\frac{d-3}{2}} \, ds \, d\phi.
	\]
	Therefore,
	\[
		\cD(\tlambda, 0)
		= 1 + 2i \hw_1 (0) |\mathbb{S}^{d-2}| \int_0^{\infty} e^{-\tlambda t} \int_0^{\infty} r^d f'(r^2) \int_{-1}^{1} e^{2it\Omega'(r^2) rs} s(1-s^2)^{\frac{d-3}{2}} \, ds \, dr \, dt,
	\]
	which completes the proof of \Cref{lem:Tdisp-extension}.
\end{proof}

The extension of $\cD(\tlambda,k)$ has been established, so \Cref{thm:Penrose-Lindhard} is equivalent to proving
\[
	\inf_{k \in \mathbb{R}^d} \inf_{\Re \tlambda \ge 0} |\cD(\tlambda, k)| \ge c_0
\]
for some $c_0 >0$. In the remainder of this section, we work on $\cD(\tlambda ,k)$ instead of $D(\lambda, k)$. We first consider the case when $\Re \tlambda >0$ and establish that there are no zeros of $\cD(\tlambda, k)$ in $\{ \Re \tlambda >0 \} \times \mathbb{R}^d$. To accomplish this, we express $\cD(\tlambda, k)$ as a Cauchy-type transform on $\mathbb{R}$.
\begin{lemma}\label{lem:disp-hku}
	Let $\gammaeq= f(-\Delta)$ be the equilibrium, and let $w_j$, $j=1,2$, be the interaction potentials. Assume that $f$ and $w_j$ satisfy \ref{equi:nonneg}--\ref{assumption:exchange}. 
	
	(i) For $\Re \tlambda >0$ and $k \neq 0$, the dispersion relation $\cD(\tlambda,k)$ can be written as
	\begin{equation}\label{def:Tdisp-hku}
	\cD(\tlambda, k) = 1- i\hw_1(k) \int_{\RR} \frac{h_k(u)}{\tlambda + iu} \, du,
	\end{equation}
	where $h_k$ is an odd function on $\mathbb{R}$ defined by
	\begin{equation}\label{def:hku}
		h_k(u) = \int_{\{ p : \widetilde{A}_{k,p} = u \}} \frac{\widetilde{a}_{k,p}}{|\nabla_p \widetilde{A}_{k,p}|} \, d \sigma(p).
	\end{equation}
	Moreover, $h_k(u) < 0$ for $u > 0$.

	(ii) For $\Re \tlambda >0$ and $k=0$, we have
	\begin{equation}\label{def:Tdisp-h0u}
		\cD(\tlambda, 0) = 1- i \hw_1(0) \int_{\RR} \frac{h_0 (u)}{\tlambda +iu} \, du,
	\end{equation}
	where $h_0$ is defined by
	\begin{equation}\label{def:h0u}
		h_0 (u) = -\int_{ \{ p : -2\Omega'(|p|^2) p_1 = u \}} \frac{f'(|p|^2)p_1}{|\nabla_p(\Omega'(|p|^2)p_1)|}  \, d\sigma(p).
	\end{equation}
	Moreover, $h_0(u) <0$ for $u >0$.
\end{lemma}
\begin{proof}
	First, we consider the case when $\Re \tlambda >0$ and $k \neq 0$. 
	Decomposing the integral with respect to the level sets of $\widetilde{A}_{k,p}$, we get
	\[
		\int_{\RR^d} \frac{\widetilde{a}_{k,p}}{\tlambda + i\widetilde{A}_{k,p}} \, dp 
		= \int_{\RR} \frac{1}{\tlambda + iu} \int_{\{ \widetilde{A}_{k,p} =u\}} \frac{\widetilde{a}_{k,p}}{|\nabla_p \widetilde{A}_{k,p}|} \, d\sigma(p) \, du = \int_{\RR} \frac{h_k(u)}{\tlambda + iu} \, du .
	\]
	It follows from \eqref{def:Tdisp} that
	\[
		\cD(\tlambda, k) = 1- i\hw_1 (k) \int_{\RR^d} \frac{\widetilde{a}_{k,p}}{\tlambda + i \widetilde{A}_{k,p}}  \, dp = 1 - i \hw_1(k) \int_{\RR} \frac{h_k(u)}{\tlambda + i u} \, du.
	\]
	Next, we show that $h_k(\cdot)$ is an odd function. Let $p'= k-p$. Since
	\[
		a_{k-p,p} = -a_{p,k-p}, \qquad A_{k-p,p} = -A_{p,k-p},
	\]
	we have
	\[
		\widetilde{a}_{k,p'} = - \widetilde{a}_{k,p}, \qquad \widetilde{A}_{k,p'} = -\widetilde{A}_{k,p}.
	\]
	Using the change of variables $p'=k-p$, it follows that
	\begin{align*}
		h_k(-u)
		= \int_{\{ \widetilde{A}_{k,p} = -u \}} \frac{\widetilde{a}_{k,p}}{|\nabla_p \widetilde{A}_{k,p}|} \, d \sigma(p)
		= - \int_{\{ \widetilde{A}_{k,p'} = u \}} \frac{\widetilde{a}_{k,p'}}{|\nabla_p \widetilde{A}_{k,p'}|} \, d \sigma(p')
		= -h_k(u).
	\end{align*}
	Finally, let $u>0$ and suppose $\widetilde{A}_{k,p}=u$. Then $A_{k-p,p} = |k|u >0$. Since $\Omega' \ge 0$ and
	\[
		A_{k-p,p} = \Omega(|k-p|^2) - \Omega(|p|^2),
	\]
	we obtain that $|k-p| > |p|$. As $f'<0$, this yields
	\[
		a_{k-p,p} = f(|k-p|^2) - f(|p|^2 ) < 0,
	\]
	which implies $\widetilde{a}_{k,p} < 0$. Hence, we conclude that $h_k(u)<0$ for $u>0$.
	
	We now consider the case when $\Re \tlambda >0$ and $k =0$. We may write the expressions in \Cref{lem:Tdisp-extension} as
	\begin{align}
		\cD(\tlambda, 0) 
		&= 1 +2i \hw_1(0) \int_{\RR^d} \frac{f'(|p|^2) p_1 }{\tlambda -2i \Omega'(|p|^2) p_1 } \, dp \label{def:Tdisp-zero1} \\
		&=1+2i\hw_1(0) |\mathbb{S}^{d-2}| \int_0^{\infty} r^d f'(r^2) \int_{-1}^{1} \frac{s(1-s^2)^{\frac{d-3}{2}}}{\tlambda-2ir\Omega'(r^2)s} \, ds \, dr. \label{def:Tdisp-zero2}
	\end{align}
	As before, we decompose the integral according to the level sets of $-2\Omega'(|p|^2) p_1$. 
	By the definition of $h_0$, we obtain
	\[
		-2 \int_{\RR^d} \frac{f'(|p|^2)p_1}{\tlambda -2i\Omega'(|p|^2)p_1} \, dp
		= \int_{\RR} \frac{h_0(u)}{\tlambda +iu} \, du.
	\]
	Therefore,
	\[
		\cD(\tlambda, 0)
		= 1- i \hw_1(0) \int_{\RR} \frac{h_0(u)}{\tlambda +iu} \, du,
	\]
	which proves \eqref{def:Tdisp-h0u}.
	Next, we show that $h_0(\cdot)$ is an odd function. 
	Write $p=(p_1, p_{\perp})$ and set $\tilde{p} = (-p_1, p_{\perp})$. Then
	\[
		-2 \Omega'(|\tilde{p}|^2) \tilde{p}_1 = 2 \Omega'(|p|^2)p_1, \qquad -2 f'(|\tilde{p}|^2) \tilde{p}_1 = 2f'(|p|^2) p_1.
	\]
	Hence, using the change of variables $p \mapsto \tilde{p}$, we obtain $h_0(-u) = - h_0 (u)$.
	Finally, let $u>0$ and suppose $-2\Omega'(|p|^2)p_1 =u$. Since $\Omega' \ge 0$, we have $p_1 < 0$. Noting that $f'(|p|^2) <0$, we obtain that $-2f'(|p|^2) p_1 <0$. Therefore, $h_0(u) <0$ for $u>0$.
\end{proof}

Now we can show that there are no growing modes of $\cD(\tlambda, k)$.
\begin{lemma}\label{lem:nonzero1}
	Let $\gammaeq= f(-\Delta)$ be the equilibrium, and let $w_j$, $j=1,2$, be the interaction potentials. Assume that $f$ and $w_j$ satisfy \ref{equi:nonneg}--\ref{assumption:exchange}.
	Then $\cD(\tlambda, k) \neq 0$ for $\Re \tlambda >0$ and $k \in \mathbb{R}^d$.
\end{lemma}
\begin{proof}
	We first consider the case when $k \neq 0$. Suppose to the contrary that $\cD(\tlambda,k) =0$ for some $\tlambda$ with $\Re \tlambda >0$. Recalling \eqref{def:Tdisp-hku}, we observe that $\cD(\tlambda, k)=0$ is equivalent to
	\begin{equation}\label{eqn:nogrowing-kneq0}
		\int_{\RR} \frac{h_k(u)}{\tlambda + iu}  \, du = - \frac{i}{\hw_1(k)},
	\end{equation}
	where $h_k(\cdot)$ is defined in \eqref{def:hku}.
	We now use the decomposition
	\[
		\frac{1}{\tlambda + iu} = \frac{\Re \tlambda}{| \tlambda+iu |^2} - i \frac{ \Im \tlambda + u}{| \tlambda+iu |^2},
	\]
	to obtain the equivalent system of equations
	\[
		\int_{\RR} \frac{h_k(u)}{|\tlambda +iu|^2} \, du =0, \qquad \int_{\RR} \frac{(\Im \tlambda + u)h_k(u)}{|\tlambda + iu|^2} \, du = - \frac{1}{\hw_1(k)},
	\]
	by comparing the real and imaginary parts of \eqref{eqn:nogrowing-kneq0}.
	It follows that
	\begin{equation}\label{eqn:nogrowing-kneq0-2}
		\int_{\RR} \frac{u h_k(u)}{|\tlambda + iu|^2} \, du = -\frac{1}{\hw_1(k)}.
	\end{equation}
	Using \Cref{lem:disp-hku}, we have $uh_k(u) \le 0$ for $u \in \mathbb{R}$, which contradicts \eqref{eqn:nogrowing-kneq0-2}.
	
	Next, we consider the case when $k=0$. As before, we assume that $\cD(\tlambda,k)=0$ for some $\tlambda$ with $\Re \tlambda >0$, and aim for a contradiction. Recalling \eqref{def:Tdisp-zero2}, 
	$\cD(\tlambda, k)=0$ is equivalent to
	\begin{equation}\label{eqn:nogrowing-keq0}
		\int_0^{\infty} r^d f'(r^2) \int_{-1}^{1} \frac{s(1-s^2)^{\frac{d-3}{2}}}{\tlambda-2ir\Omega'(r^2)s} \, ds \, dr = \frac{i}{2 \hw_1(0) |\mathbb{S}^{d-2}|}.
	\end{equation}
	We have the decomposition
	\[
		\frac{1}{\tlambda-2ir\Omega'(r^2)s} = \frac{\Re \tlambda}{|\tlambda-2ir\Omega'(r^2)s|^2} + i \frac{2r\Omega'(r^2)s - \Im \tlambda}{|\tlambda-2ir\Omega'(r^2)s|^2},
	\]
	which yields the equivalent system
	\[
		\int_0^{\infty} r^d f'(r^2) \int_{-1}^{1} \frac{s(1-s^2)^{\frac{d-3}{2}}}{|\tlambda-2ir\Omega'(r^2)s|^2} \, ds \, dr =0,
	\]
	and
	\[
		\int_0^{\infty} r^d f'(r^2) \int_{-1}^{1} \frac{(2r\Omega'(r^2)s - \Im \tlambda)s(1-s^2)^{\frac{d-3}{2}}}{|\tlambda-2ir\Omega'(r^2)s|^2} \, ds \, dr = \frac{1}{2\hw_1(0)|\mathbb{S}^{d-2}|}.
	\]
	by evaluating the real and imaginary parts of \eqref{eqn:nogrowing-keq0}. It follows that
	\begin{equation}\label{eqn:nogrowing-keq0-2}
		\int_0^{\infty} 2r^{d+1} f'(r^2) \Omega'(r^2) \int_{-1}^{1} \frac{s^2(1-s^2)^{\frac{d-3}{2}}}{|\tlambda-2ir\Omega'(r^2)s|^2} \, ds \, dr = \frac{1}{2\hw_1(0)|\mathbb{S}^{d-2}|},
	\end{equation}
	which is a contradiction since $f'(r^2) \Omega'(r^2)  \le 0$. This ends the proof of \Cref{lem:nonzero1}.
\end{proof}

Next, we show that $\cD(\tlambda, k)$ does not vanish on $\{ \Re \tlambda =0\} \times \mathbb{R}^d$.
\begin{lemma}\label{lem:nonzero2}
	Let $\gammaeq= f(-\Delta)$ be the equilibrium, and let $w_j$, $j=1,2$, be the interaction potentials. Assume that $f$ and $w_j$ satisfy \ref{equi:nonneg}--\ref{assumption:exchange}. We have $\cD(i\ttau, k) \neq 0$ for all $\ttau \in \mathbb{R}$ and $k \in \mathbb{R}^d$.
\end{lemma}
\begin{proof}
	We first consider the case when $k\neq 0$. Recall from \eqref{def:Tdisp-hku} that
	\[
		\cD(\tlambda, k) = 1- i\hw_1(k) \int_{\RR} \frac{h_k(u)}{\tlambda +iu} \, du
	\]
	for $\Re \tlambda >0$. Noting that $h_k(u) \neq 0$ for all $u \neq 0$, we apply Plemelj's formula to obtain for $\ttau \in \mathbb{R}$ that
	\begin{align*}
		\cD(i\ttau, k)
		&= \lim_{\widetilde{\sigma} \to 0^{+}} \cD (\widetilde{\sigma} + i\ttau , k) \\
		&= 1- i\hw_1(k) \lim_{\widetilde{\sigma} \to 0^{+}} \int_{\RR} \frac{h_k(u)}{\widetilde{\sigma} + i(\ttau + u)} \, du \\
		&= 1 - \hw_1(k) \left[ \operatorname{PV}\int_{\RR} \frac{h_k(u)}{u + \ttau} \, du + i\pi h_k (-\ttau) \right].
	\end{align*}
	Therefore, we obtain a lower bound by looking at its imaginary part, namely,
	\[
		|\cD(i\ttau, k)| \ge \pi \hw_1(k) |h_k(-\ttau)|.
	\]
	From \Cref{lem:disp-hku}, we have $h_k(-\ttau)$ for $\ttau \neq 0$, which implies that $\cD(i\ttau, k) \neq 0$ for $\ttau \neq 0$. It remains to check at $\ttau =0$. We have
	\[
		\cD(0, k) = 1 - \hw_1(k) \operatorname{PV} \int_{\RR} \frac{h_k(u)}{u} \, du = 1- 2 \hw_1(k) \int_{0}^{\infty} \frac{h_k(u)}{u} \, du >1
	\]
	since $h_k(u) <0$ for $u>0$.
	
	Next, we consider the case when $k=0$. Applying Plemelj's formula for \eqref{def:Tdisp-h0u}, we obtain for $\ttau \in \RR$ that
	\begin{align*}
		\cD(i\ttau, 0)
		&= \lim_{\widetilde{\sigma} \to 0^{+}} \cD (\widetilde{\sigma} + i\ttau , 0) \\
		&= 1- i\hw_1(0) \lim_{\widetilde{\sigma} \to 0^{+}} \int_{\RR} \frac{h_0(u)}{\widetilde{\sigma} + i(\ttau + u)} \, du \\
		&= 1 - \hw_1(0) \left[ \operatorname{PV}\int_{\RR} \frac{h_0(u)}{u + \ttau} \, du + i\pi h_0 (-\ttau) \right].
	\end{align*}
	It follows from \Cref{lem:disp-hku} that
	\[
		|\cD(i\ttau, k) | \ge \pi \hw_1(0) |h_0(-\ttau)| >0
	\]
	for $\ttau \neq 0$. Meanwhile, we observe that
	\[
		\cD(0,0) = 1- \hw_1(0) \operatorname{PV} \int_{\RR} \frac{h_0(u)}{u} \, du = 1- 2 \hw_1(0) \int_0^{\infty} \frac{h_0(u)}{u} \, du >1,
	\]
	noting that $h_0(u) < 0$ for $u>0$. This concludes the proof of \Cref{lem:nonzero2}.
\end{proof}

Now we complete the proof of \Cref{thm:Penrose-Lindhard}.
\begin{proof}[Proof of \Cref{thm:Penrose-Lindhard}]
	We would like to prove
	\[
		\inf_{k \in \mathbb{R}^d} \inf_{ \Re \tlambda \ge 0} | \cD(\tlambda, k)|  \gtrsim 1.
	\]
For $\Re \tlambda \ge 0$, we recall from \eqref{def:Tdisp-J}, \eqref{def:Tdisp-timeint} and the proof of \Cref{lem:Tdisp-extension} that
\[
	|\cD(\tlambda, k) - 1| \le  |\hw_1(k)| \int_0^{\infty} |e^{-\tlambda t} \cJ(t,k)| \, dt \lesssim |\hw_1(k)| \langle k \rangle^{-1},
\]
which converges to zero as $|k| \to \infty$. Hence there exists $K>0$ such that $|\cD(\tlambda, k)| \ge 1/2$ uniformly for $|k| \ge K$ and $\Re \tlambda \ge 0$.

For $|k| \le K$, $\Re \tlambda \ge 0$ with $|\tlambda| \gg 1$, we integrate by parts in $t$ for two times to get
\begin{equation}\label{eqn:J-IBP-t}
	\int_0^{\infty} e^{-\tlambda t} \cJ(t,k) \, dt
	= \frac{1}{\tlambda^2} \partial_t \cJ(0,k) + \frac{1}{\tlambda^2} \int_0^{\infty} e^{-\tlambda t} \partial_t^2 \cJ(t,k) \, dt,
\end{equation}
noting that $\cJ(0,k) =0$. 
From the definition \eqref{def:Tdisp-J} of $\cJ(t,k)$, we compute for $k \neq 0$ that
\begin{align*}
	\partial_t \cJ(t,k) = -i \int_{\RR^d} e^{-it\widetilde{A}_{k,p}} \widetilde{A}_{k,p} \widetilde{a}_{k,p} \, dp, \qquad
	\partial_t^2 \cJ(t,k) = -\int_{\RR^d} e^{-it\widetilde{A}_{k,p}} \widetilde{A}_{k,p}^2 \widetilde{a}_{k,p} \, dp .
\end{align*}
Noting that $|k| \lesssim 1$, we obtain from \Cref{lem:Akp,lem:akp} that
\[
|\widetilde{A}_{k,p}| \lesssim  \langle p \rangle, \qquad |\partial_p \widetilde{A}_{k,p}| \lesssim 1, \qquad |\partial_p^2 \widetilde{A}_{k,p}| \lesssim 1, 
\]
and $|\partial_p^{\alpha} \widetilde{a}_{k,p} | \lesssim  \langle p \rangle^{-2n_1}$ for $|\alpha| \le 2$. Then
\begin{equation}\label{est:dtJ0}
	|\partial_t \cJ(0,k)| \lesssim  \| \widetilde{A}_{k,p} \widetilde{a}_{k,p} \|_{L^1_p} \lesssim 1
\end{equation}
uniformly for $k \neq 0$, provided that $n_1 >\frac{d+1}{2}$. Next,  integrating by parts in $p$ yields
\begin{equation}\label{est:d2tJt}
	|\partial_t^2 \cJ(t,k)|  \lesssim \langle t \rangle^{-2}  \| \widetilde{A}_{k,p}^2 \, \widetilde{a}_{k,p} \|_{W^{2,1}_p} \lesssim \langle t \rangle^{-2}
\end{equation}
uniformly for $k \neq 0$, since the contribution of the worst term is
\[
	 |\widetilde{A}_{k,p}^2 \, \partial_p^2 \widetilde{a}_{k,p}| \lesssim \langle p \rangle^{2-2n_1},
\]
from which we require $n_1 > \frac{d+2}{2}$ to ensure that $\| \widetilde{A}_{k,p} \widetilde{a}_{k,p}  \|_{W^{2,1}_{p}} \lesssim 1$. Combining \eqref{est:dtJ0} and \eqref{est:d2tJt}, we obtain from \eqref{eqn:J-IBP-t} that
\[
|\cD(\tlambda, k) -1 | \lesssim \frac{|\hw_1(k)|}{|\tlambda|^2} \left( |\partial_t \cJ(0,k)| + \int_0^{\infty} |\partial_t^2 \cJ(t,k)| \, dt \right)
\]
uniformly for $|k| \le K$ and $\Re \tlambda \ge 0$. Note that the case when $k=0$ follows from the continuity of $\cD(\tlambda ,k)$.
The right hand side converges to $0$ as $|\tlambda| \to \infty$. Hence there exists $\Lambda>0$ such that $|\cD(\tlambda, k)| \ge 1/2$ uniformly for $|k| \le K$, $\Re \tlambda \ge 0$, and $|\tlambda| \ge \Lambda$.

It remains to consider the region
\[
	\{ (\tlambda, k):  \Re \tlambda \ge 0, \ |\tlambda| \le \Lambda, \ |k| \le K \}.
\]
By compactness, it is enough to show $\cD(\tlambda, k) \neq 0$ for $\Re \tlambda \ge 0$, $k\in \mathbb{R}^d$. This is done in \Cref{lem:nonzero1,lem:nonzero2}.
\end{proof}

\subsection{Green function}

In this section, we shall derive pointwise bounds on the associated Green function to the linearized problem defined by inverting the Laplace transform
\begin{equation}\label{defGreenk}
	\widehat{G}_k(t) = \frac{1}{2\pi i} \int_{\{ \Re \lambda = \gamma_{0} \}} e^{\lambda t } \widetilde{G}(\lambda,k) \, d \lambda, \qquad \widetilde{G}(\lambda, k) := \frac{1}{D(\lambda, k)} 
\end{equation}
where $\gamma_{0}>0$, noting that $\widetilde{G}(\lambda, k)$ does not have any singularities in $\{ \Re \lambda > 0 \}$ for $k \neq 0$.
Using the regularity and decay assumptions of $f$, we obtain the decay of the Green function in Fourier space.
\begin{proposition}\label{prop:Green}
	Let $\gammaeq= f(-\Delta)$ be the equilibrium, and let $w_j$, $j=1,2$, be the interaction potentials. Assume that $f$ and $w_j$ satisfy \ref{equi:nonneg}--\ref{assumption:exchange}. Then the Green function $\widehat{G}(t,k)$ defined as in \eqref{defGreenk} can be decomposed into
	\begin{equation}\label{decomp:Green}
		\widehat{G}(t,k) = \delta(t) + \widehat{G}^r(t,k)
	\end{equation}
	where $\delta(t)$ is the Dirac delta distribution, and the regular part  $\widehat{G}^r(t,k)$ satisfies
	\begin{equation}
		\label{est:Green-Fourier}
		|\widehat{G}^r(t,k)|
		\le C_{n_0} |k| \langle kt \rangle^{-n_0 +3},
	\end{equation}
	for all $t \ge 0$, $k\in \mathbb{R}^d \setminus \{0\}$, and for some constant $C_{n_0}$ depending only on $n_0$.
\end{proposition}

\begin{proof}
	By definition, we may write 
	\[
		\widetilde{G}(\lambda, k) = 1+ \frac{1- D(\lambda,k)}{D(\lambda, k)} 
		= 1 - \frac{\hw_1(k) m_f(\lambda, k)}{1+ \hw_1(k) m_f (\lambda, k)} =: 1 + \widetilde{G}^r(\lambda, k)
	\]
	for $\Re \lambda \ge 0$ and $k \neq 0$. Here we define
	\[
		m_f (\lambda, k) := -i \int_0^{\infty} e^{-\lambda t /|k|} \int_{\RR^d} e^{-it\widetilde{A}_{k,p}} \widetilde{a}_{k,p} \, dp \, dt
		= -i \int_0^{\infty} e^{-\lambda t /|k|} \cJ(t,k) \, dt,
	\]
	recalling \eqref{def:Tdisp-J}. 
	We claim that
	\begin{align}
		|\cJ(t,k)| &\lesssim \langle k \rangle^{-1} \langle t \rangle^{-n_0 +1},  \label{est:J-timedecay1}\\
		|\partial_t \cJ(t,k) | &\lesssim \langle t \rangle^{-n_0 +1}, \label{est:J-timedecay2}\\
		|\partial_t^2 \cJ(t,k)| &\lesssim \langle t \rangle^{-n_0 +1}, \label{est:J-timedecay3}
	\end{align}
	provided that $n_1 > \frac{d+2}{2}$.
	We first note that \Cref{lem:Akp} yields
	\[
		|\widetilde{A}_{k,p}| \lesssim  |k| + \langle p \rangle, \qquad |\partial_p \widetilde{A}_{k,p}| \sim 1, \qquad |\partial_p^{\alpha} \widetilde{A}_{k,p}| \lesssim 1
	\]
	for $1 \le |\alpha| \le 3n_0 -1$ and \Cref{lem:akp} gives
		\[
		|\partial_p^{\alpha} \widetilde{a}_{k,p}| \lesssim \langle k \rangle^{-1} \left( \langle k -p \rangle^{-2n_1} + \langle p \rangle^{-2n_1} \right)
		\]
	for $|\alpha| \le n_0 -1$.
	It follows that
	\[
		\| \widetilde{a}_{k,p} \|_{W^{n_0 -1, 1}_p} \lesssim \langle k \rangle^{-1}, \qquad \| \widetilde{A}_{k,p} \widetilde{a}_{k,p} \|_{W^{n_0 -1, 1}_p} \lesssim 1, \qquad \| \widetilde{A}_{k,p}^2 \widetilde{a}_{k,p} \|_{W^{n_0 -1, 1}_p} \lesssim  \langle k \rangle,
	\]
	where we use $n_1 > \frac{d+2}{2}$.
	A direct computation yields
	\begin{align*}
		\cJ(t,k) &= \int_{\RR^d} e^{-it\widetilde{A}_{k,p}} \widetilde{a}_{k,p} \, dp, \\
		\partial_t \cJ(t,k) &= -i \int_{\RR^d} e^{-it\widetilde{A}_{k,p}} \widetilde{A}_{k,p} \widetilde{a}_{k,p} \, dp, \\
		\partial_t^2 \cJ(t,k) &= \int_{\RR^d} e^{-it\widetilde{A}_{k,p}} \widetilde{A}_{k,p}^2 \widetilde{a}_{k,p} \, dp.
	\end{align*}
	Integrating by parts in $p$ for $n_0 -1$ times, we get
	\begin{align*}
		|\cJ(t,k)| 
		\lesssim  \left| \int_{\RR^d} e^{-it\widetilde{A}_{k,p}} \widetilde{a}_{k,p} \, dp \right|
		\lesssim \langle t \rangle^{-n_0 + 1} \| \widetilde{a}_{k,p} \|_{W^{n_0 -1 ,1}_p}
		\lesssim \langle k \rangle^{-1} \langle t \rangle^{-n_0 +1}.
	\end{align*}
	uniformly in $k$, which proves \eqref{est:J-timedecay1}. Similarly, integrating by parts in $p$ gives
	\[
		|\partial_t \cJ(t,k)| \lesssim \langle t \rangle^{-n_0 + 1} \| \widetilde{A}_{k,p} \widetilde{a}_{k,p} \|_{W^{n_0 -1, 1}_p} \lesssim \langle t \rangle^{-n_0 +1},
	\]
	and
	\[
		|\partial_t^2 \cJ(t,k)| \lesssim \langle t \rangle^{-n_0 +1} \| \widetilde{A}_{k,p}^2 \widetilde{a}_{k,p} \|_{W^{n_0 -1, 1}_p} \lesssim \langle k \rangle \langle t \rangle^{-n_0 +1},
	\]
	which prove \eqref{est:J-timedecay2} and \eqref{est:J-timedecay3}, respectively.
	
	We now estimate $m_f(\lambda, k)$. For $\Re \lambda \ge 0$, we obtain from \eqref{est:J-timedecay1} that
	\begin{equation*}
		|m_f(\lambda, k)| \lesssim \int_0^{\infty} |\cJ(t,k)| \, dt \lesssim \langle k \rangle^{-1} \int_0^{\infty} \langle t \rangle^{-n_0 +1} \, dt \lesssim \langle k \rangle^{-1},
	\end{equation*}
	noting that $n_0 >2$.
	On the other hand, for $\Re \lambda =0$, we write $\lambda = i\tau$ and integrate by parts twice in $t$ to get
	\[
		(|k|^2 \langle k \rangle^2 + \tau^2) m_f (i\tau, k)
		= |k|^2 \int_{\RR^d} \widetilde{A}_{k,p} \widetilde{a}_{k,p}\, dp -i |k|^2 \int_0^{\infty} e^{-i\tau t/|k|} (\langle k \rangle^2 - \partial_t^2) \cJ(t,k) \, dt,
	\]
	noting that the boundary terms at $t=\infty$ vanish from \eqref{est:J-timedecay1}, \eqref{est:J-timedecay2}, and we also used
	\[
		\cJ(0,k) =0 , \qquad \partial_t \cJ(0,k) = -i\int_{\RR^d} \widetilde{A}_{k,p} \widetilde{a}_{k,p}\, dp.
	\]
	Hence, for $\Re \lambda =0$, we have
	\[
		|m_f(\lambda, k)|
		\lesssim  \frac{|k|^2}{|k|^2 \langle k \rangle^2 + |\Im \lambda|^2} + |k|^2 \int_0^{\infty} \frac{\langle k \rangle}{|k|^2 \langle k \rangle^2 + |\Im \lambda|^2} \langle t \rangle^{-n_0 + 1} \, dt
		\lesssim \frac{|k|^2 \langle k \rangle}{|k|^2 \langle k \rangle^2 + |\Im \lambda|^2}.
	\]
	We apply the same argument to $\partial_\lambda^n m_f(\lambda, k)$ using $\partial_{\lambda}^n e^{-\lambda t/|k|} = (-t/|k|)^n e^{-\lambda t/|k|}$ for $ 0 < n <n_0-2$. 
	For $\Re \lambda \ge 0$, we obtain
	\[
		|\partial_\lambda^n m_f(\lambda, k)| \lesssim |k|^{-n} \int_0^{\infty} t^n |\cJ(t,k)| \, dt 
		\lesssim |k|^{-n} \langle k \rangle^{-1} \int_0^{\infty} t^n \langle t \rangle^{-n_0 +1} \, dt 
		\lesssim |k|^{-n} \langle k \rangle^{-1}.
	\]
	Similarly, for $\Re \lambda =0$, we get
	\begin{align*}
		|\partial_{\lambda}^n m_f(\lambda,k)|
		&\lesssim \frac{|k|^{2-n} }{|k|^2 \langle k \rangle^2 + |\Im \lambda|^2} \left( \langle k \rangle \int_0^{\infty} t^n |\cJ(t,k) | \, dt + \int_0^{\infty} |\partial_t^2 (t^n \cJ(t,k)) | \, dt \right)  \\
		&\lesssim \frac{|k|^{2-n} }{|k|^2 \langle k \rangle^2 + |\Im \lambda|^2} \left( \int_0^{\infty} t^n \langle t \rangle^{-n_0 +1} \, dt + \langle k \rangle \int_0^{\infty}t^n \langle t \rangle^{-n_0 +1} \, dt \right) \\
		&\lesssim  \frac{|k|^{2-n} \langle k \rangle}{|k|^2 \langle k \rangle^2 +|\Im \lambda|^2}.
	\end{align*}
	Using \Cref{thm:Penrose-Lindhard} and $|\hw_1 (k) |\lesssim 1$, we obtain that there exists a positive constant $C_{n_0}$, independent of $\lambda$, $k$, and $n$, satisfying
	\[
		|\partial_{\lambda}^n \widetilde{G}^r(\lambda,k)|
		\le \left| \sum_{j=0}^n \binom{n}{j} \partial_{\lambda}^j (1-D(\lambda,k)) \partial_{\lambda}^{n-j} (D(\lambda,k))^{-1} \right|
		\le \frac{C_{n_0} |k|^{2-n} \langle k \rangle}{|k|^2 \langle k \rangle^2 + |\Im \lambda|^2} 
	\]
	for $0 \le n < n_0 -2$ and $\Re \lambda \ge 0$. Therefore, upon integrating by parts, it follows that
	\[
		\widehat{G}^r(t,k) = \frac{1}{2\pi i}  \int_{\{\Re \lambda =0 \}} e^{\lambda t} \widetilde{G}^r (\lambda,k) \, d\lambda = \frac{1}{2\pi} \int_{\RR} e^{i\tau t}\widetilde{G}^r(i\tau,k) \, d\tau = \frac{(-1)^n}{2\pi t^n } \int_{\RR} e^{i\tau t} \partial_{\lambda}^n \widetilde{G}^r(i\tau,k) \, d\tau
	\]
	for all $ 0 \le n < n_0 -2$. We use $\int_{\RR} (a^2 + x^2)^{-1} \, dx = \pi |a|^{-1}$. For $n=0$, we obtain
	\[
		|\widehat{G}^r (t,k)| \le C_{n_0} \int_{\RR} \frac{|k|^2 \langle k \rangle}{|k|^2\langle k \rangle^2 + \tau^2} \, d \tau \lesssim |k| .
	\]
	For $n= n_0 -3$, we obtain
	\[
		|\widehat{G}^r (t,k) | \le \frac{C_{n_0}}{t^{n_0 -3}} \int_{\RR} \frac{|k|^{-n_0 + 5} \langle k \rangle}{|k|^2 \langle k \rangle^2 + \tau^2} \, d\tau
		\le C_{n_0} |k|  |kt|^{-n_0 +3}.
	\]
	Combining two bounds, we get \eqref{est:Green-Fourier}, which concludes the proof of \Cref{prop:Green}.
\end{proof}


\section{Nonlinear analysis}\label{sec-iter}
In this section, we return to the full nonlinear Hartree--Fock equation \eqref{eqn:perturb}, and prove the main results.

\subsection{Nonlinear iterative scheme}

Specifically, fix a small constant $\delta>0$ and any constants $\sigma>2d$ and $N>d$ so that $N \le n_0 -3$ and $\sigma \le 2n_1$, where $n_0,n_1$ are defined as in Assumption {\ref{equi:reg}}. 
For any $t\ge 0$, introduce the following iterative norm
\begin{equation}\label{key-iter}
\zeta(t) = \zeta_\rho(t) + \zeta_\mu(t) 
\end{equation}
where 
 \begin{equation}\label{defzeta}
\begin{aligned}
\zeta_\rho(t): &= \sup_{0\le s\le t} \Big\{  \langle s\rangle^{-\delta} \| \langle k \rangle^{\sigma} \langle ks \rangle^{N} \widehat{\rho}_\gamma(s,k) \|_{L^{\infty}_k} + \epsilon_1^{-1}  \langle s\rangle^{-\delta}\sum_{|\alpha| \le N} \| \langle k \rangle^{\sigma} \langle ks \rangle^{N} \partial_p^{\alpha} \widehat{\cX}_\gamma(s,k-p,p) \|_{L^{\infty}_{k,p}} \Big\}\\
\zeta_\mu(t): &=\sup_{0\le s\le t} \Big\{  \|\Fmu(s)\|_{\cH^{N-2}} +  \langle s\rangle^{-\delta} \|\Fmu(s)\|_{\cH^{N-1}_*} + \langle s\rangle^{-1-\delta} \|\Fmu(s)\|_{\cH^N_*}
	\Big\},
\end{aligned}
\end{equation}
where, for $m\ge 0$,  
\begin{equation}\label{defHnorm} 
\begin{aligned}
\|\Fmu(t)\|_{\cH^m} &= \sum_{|\alpha| \le m} \| \min \{ 1, |k|^{\delta}\}  \langle k\rangle^{\sigma+1} \min\{\langle k-p \rangle,\langle p \rangle\}^{\sigma}\partial_p^{\alpha} \Fmu(t,k-p,p) \|_{L^{\infty}_{k,p}},
\\
\|\Fmu(t)\|_{\cH^m_*} &= \sum_{|\alpha| \le m} \| \min \{ 1, |k|^{\delta}\}  \langle k\rangle^{\sigma} \min\{\langle k-p \rangle,\langle p \rangle\}^{\sigma}\partial_p^{\alpha} \Fmu(t,k-p,p) \|_{L^{\infty}_{k,p}}.
\end{aligned}\end{equation}
We note that the only difference between $\cH^m$ and $\cH^m_*$ lies in the weight $\langle k \rangle^{\sigma+1}$ versus $\langle k\rangle^\sigma$. In particular, we clearly have $\| \Fmu(t)\|_{\cH^m_*} \lesssim \| \Fmu(t)\|_{\cH^m}$. In addition, the exchange term $ \partial_p^{\alpha} \widehat{\cX}_\gamma(s,k-p,p)$ satisfies the same bounds as those for the density $\Frho(s,k)$, up to a small constant prefactor $\epsilon_1$. Observe that the boundedness of $\zeta(t)$ immediately yields uniform bounds on the density and the Hartree--Fock solution: precisely, each term in the above bracket is bounded by $\zeta(t)$. By the standard local existence theory, the iterative norm $\zeta(t)$ exists and is bounded by $C_0\epsilon $ for some sufficiently small time $t>0$. In order to propagate the iterative norm globally in time, it suffices to derive the following a priori bounds. 

\begin{proposition}\label{prop-iter} Introduce the iterative norm $\zeta(t)$ as in \eqref{key-iter}. As long as $\zeta(t)$ remains finite, there hold
\begin{equation}\label{apriori}
\begin{aligned}
\zeta(t) &\le C_0 \epsilon_0  + C_1 \epsilon_1 \zeta(t)+ C_1 \zeta(t)^2
\end{aligned}\end{equation}
for any $t\ge 0$ and for some universal constants $C_0, C_1$ that are independent of $t\ge 0$.
\end{proposition}

Provided the validity of Proposition \ref{prop-iter}, we obtain the boundedness of $\zeta(t) \le 2C_0\epsilon_0$ for sufficiently small constants $\epsilon_0,\epsilon_1$, and therefore, the global existence theory follows via the standard continuous induction, see, e.g., \cite{NguyenYou2024, You2024} for a related scheme. It remains to establish Proposition \ref{prop-iter}. 

\subsection{Nonlinear estimates}

In this section, we prove the boostrap assumptions on the profile $\Fmu(t,k,p)$. Specifically, we obtain the following. 

\begin{proposition}\label{prop-profile} 
As long as $\zeta(t)$ remains finite, there hold 
\begin{equation}\label{estFmu}
\begin{aligned}
\|\Fmu(t)\|_{\cH^{|\alpha|}} &\le C_\alpha(t) \Big(C_0\epsilon_0 + C_1 \zeta_\rho(t)+ C_2\zeta(t)^2\Big), \qquad \mbox{if}\quad |\alpha|\le N-2,
\\
\|\Fmu(t)\|_{\cH^{|\alpha|}_*} &\le C_\alpha(t) \Big(C_0\epsilon_0 + C_1 \zeta_\rho(t)+ C_2\zeta(t)^2\Big), \qquad \mbox{if}\quad |\alpha| = N-1 \quad \mbox{or}\quad |\alpha|=N,
\end{aligned}
\end{equation}
for some universal constants $C_0,C_1,C_2$, where $\|\cdot\|_{\cH^{|\alpha|}}$ and $\|\cdot\|_{\cH^{|\alpha|}_*}$ norms are defined as in \eqref{defHnorm}, and 
\begin{equation}\label{defCalpha}
C_\alpha(t): = \left\{ 
\begin{aligned} 
1, \qquad & \mbox{if} \qquad |\alpha|\le N-2,
\\
\langle t\rangle^\delta, \qquad & \mbox{if} \qquad |\alpha| = N-1,
\\
\langle t\rangle^{1+\delta}, \qquad & \mbox{if} \qquad |\alpha| = N,
\end{aligned}\right.
\end{equation}
where $\delta>0$ is fixed as in \eqref{defzeta}.
\end{proposition} 

\begin{proof} 
Recall that by definition \eqref{defHnorm}, $\|\Fmu(t)\|_{\cH^m_*} \le \|\Fmu(t)\|_{\cH^m}$, and so the second estimate in \eqref{estFmu} concerns only the top derivatives $\partial_p^\alpha \Fmu$ with $|\alpha| = N-1$ or $|\alpha| = N$.
Integrating \eqref{eqn:profile} in time and evaluating the result at $(k-p,p)$, we obtain 
$$
\begin{aligned}
	\Fmu(t,k-p,p) 	&= \Fgamma_0 (k-p,p) -i \int_0^t e^{isA_{k-p,p}} a_{k-p,p}\widehat{\cP}_{\gamma}(s,k-p,p) \, ds \\
	&\quad +i (2\pi)^{-d} \int_0^t \int e^{isA_{k-p,\ell-k+p}} \widehat{\cP}_{\gamma}(s, k-p, \ell-k+p)
\widehat{\mu}(s,k-p-\ell, p)  \, d \ell \, ds\\
	&\quad -i (2\pi)^{-d} \int_0^t \int e^{-isA_{p,\ell-p}}  \widehat{\cP}_{\gamma}(s, \ell-p, p) \widehat{\mu}(s, k-p,p-\ell)  \, d\ell \, ds.
\end{aligned}$$	
For convenience, we write $$\Fmu(t,k-p,p)
= \Fmu^{(0)} + \Fmu^{(1)} + \Fmu^{(2)} +\Fmu^{(3)},$$
where $\Fmu^{(j)}$ denote the respective terms in the above expression for $\Fmu(t,k-p,p)$. Let us bound each term on the right. Indeed, by construction, for $|\alpha|\le N$, we bound 
$$
|\partial_p^\alpha \Fmu^{(0)}| \le |\partial_p^\alpha \Fgamma_0 (k-p,p)| \lesssim \epsilon_0 \langle k-p\rangle^{-\sigma_0}\langle p\rangle^{-\sigma_0} \lesssim \epsilon_0 \langle k\rangle^{-\sigma_0} \min\{\langle k-p \rangle,\langle p \rangle\}^{-\sigma_0}
$$
which yields $\|\Fmu^{(0)}\|_{\cH^N} \lesssim \epsilon_0$ as desired, provided $\sigma_0\ge \sigma+1$, upon recalling the {$\cH$-norm} in \eqref{defHnorm}. Next, we recall from \eqref{comFcPinv} that 
\begin{equation}\label{comFcPre}
\FcP_\gamma(t,k-p,p) = \hw(k) \widehat{\rho}_{\gamma}(t,k) - \widehat{\mathcal{X}}_{\gamma}(t,k-p,p).
\end{equation}
Observe that $p$-derivatives of $\FcP_\gamma(t,k-p,p) $ does not hit the space density $\Frho_\gamma(t,k)$ for $|\alpha|\ge 1$. In addition, 
in view of the boostrap assumptions in \eqref{defzeta}, the $p$-derivatives of the exchange term $\partial_p^\alpha\widehat{\cX}_\gamma (s,k-p,p)$ satisfy the same bounds as those for the meanfield term $\hw_1(k)\Frho(t,k)$ (noting $|\hw_1(\ell)|\lesssim 1$), up to a constant prefactor $\epsilon_1$. Namely, we have 
\begin{equation}\label{rhoXchange}
\begin{aligned}
|\partial_p^\alpha\widehat{\cP}_\gamma (s,\ell-p,p)|&\le \zeta_\rho(s)\langle \ell\rangle^{-\sigma} \langle \ell s\rangle^{-N}\langle s\rangle^{\delta},
\\
|\partial_p^\alpha\widehat{\cP}_\gamma (s,k-p,\ell-k+p)| 
&\le  \zeta_\rho(s)\langle \ell\rangle^{-\sigma} \langle \ell s\rangle^{-N}\langle s\rangle^{\delta},
\end{aligned}
\end{equation}
uniformly in $k,p$ and $s\ge 0$, for $|\alpha|\le N$. 

\subsubsection*{Bounds on $\Fmu^{(1)}$.}

By construction, for $|\alpha|\le N$, we compute 
$$
\begin{aligned}
\partial_p^\alpha \Fmu^{(1)} 
&= -\sum_{\alpha'+{\alpha''}+ \alpha'''= \alpha}\int_0^t \partial_p^{\alpha'} e^{isA_{k-p,p}} \partial_p^{\alpha''} a_{k-p,p} \partial_p^{\alpha'''}\widehat{\cP}_{\gamma}(s,k-p,p) \, ds.
\end{aligned}$$
Using $\partial_p A_{k-p,p} \sim k$ and $|\partial_p^{\beta} A_{k-p,p}| \lesssim |k|$ for $|\beta|\ge2$, see Lemma \ref{lem:Akp}, we bound 
$$ |\partial_p^{\alpha'} e^{isA_{k-p,p}} | \lesssim \langle ks\rangle^{|{\alpha'}|}.$$
Therefore, using the bootstrap assumptions \eqref{rhoXchange}, we obtain 
$$
\begin{aligned}
|\partial_p^\alpha \Fmu^{(1)}| 
&\lesssim \zeta_\rho(t) \sum_{{\alpha'}+{\alpha''} \le \alpha}\int_0^t \langle s\rangle^\delta \langle k \rangle^{-\sigma}\langle ks\rangle^{-N+|{\alpha'}|} 
|\partial_p^{\alpha''} a_{k-p,p} | \, ds.
\end{aligned}$$
Using \eqref{dervakp}, we get 
\begin{align*}
	|k|^{-1} |\partial_p^{\alpha''} a_{k-p,p}|
	&\lesssim \langle k \rangle^{-1} \left( \langle k-p \rangle^{-2n_1} + \langle p \rangle^{-2n_1} \right) \\
	&\lesssim \langle k \rangle^{-1} \min \{ \langle k-p \rangle, \langle p \rangle \}^{-2n_1},
\end{align*}
which gives the right decay in $k$ and $\min \{ \langle k -p \rangle, \langle p \rangle\}$, provided $2n_1 \ge \sigma$.
Hence, we only need to bound
\begin{equation}
	\begin{split}
		\int_0^t \langle s \rangle^{\delta} \langle ks \rangle^{-N + |\alpha'|} |k| \, ds	
		\lesssim 
		\begin{cases}
			\max \{ 1, |k|^{-\delta}\},	\quad				&\textrm{if} \quad |\alpha'| \le N-2, \\
			C_{\alpha'}(t) \langle k \rangle	, \quad &\mbox{if}\quad |\alpha'| = N-1 \quad \mbox{or}\quad |\alpha'|=N,		\end{cases}
	\end{split}
\end{equation}
with $C_\alpha(t)$ defined in \eqref{defCalpha}. When $|\alpha'| \le N-2$, using $\langle s \rangle^{\delta} \le \langle ks \rangle^{\delta} \max \{ 1, |k|^{-\delta}\}$, we bound
\begin{align*}
	\int_0^t \langle s \rangle^{\delta} \langle ks \rangle^{-2} |k| \, ds
	&\lesssim \max \{ 1, |k|^{-\delta}\} \int_0^t \langle ks \rangle^{-2+\delta} |k| \, ds
	\lesssim \max \{ 1, |k|^{-\delta}\},
\end{align*}
as desired. Similarly, when $|\alpha'| = N-1$, we bound
\begin{align*}
	\int_0^t \langle s \rangle^{\delta} \langle ks \rangle^{-2} |k| \, ds
	&\lesssim \max \{ 1, |k|^{-\delta}\} \int_0^t \langle ks \rangle^{-1+\delta} |k| \, ds 
	\\
&\lesssim \max \{ 1, |k|^{-\delta}\} \int_0^t s^{-1+\delta} |k|^\delta \, ds 
	\lesssim \langle t \rangle^{\delta}\langle k\rangle^{\delta}.
\end{align*}
When $|\alpha'|=N$, we bound
\begin{align*}
	\int_0^t \langle s \rangle^{\delta} |k| \, ds
	&\lesssim \langle t \rangle^{1+\delta} |k|.
\end{align*}
This proves that $\|\Fmu^{(1)}\|_{\cH^{|\alpha|}} \le C_\alpha(t) \zeta_\rho(t)$ as desired, for all $|\alpha|\le N$.

\subsubsection*{Bounds on $\Fmu^{(2)}$.}
Next, we bound the nonlinear contribution $\Fmu^{(2)}$. By construction, we compute 
\begin{equation}\label{defmu212}
\begin{aligned}
\partial_p^\alpha \Fmu^{(2)}
&= 
\sum_{{\alpha'} + {\alpha''} +\alpha'''= \alpha}\int_0^t \int \partial_p^{\alpha'} e^{isA_{k-p,\ell-k+p}} \partial_p^{\alpha''}\widehat{\mu}(s,k-p-\ell, p)\partial_p^{\alpha'''}\widehat{\cP}_{\gamma}(s, k-p, \ell-k+p) \; d\ell \, ds.
\end{aligned}
\end{equation}
By definition of the {$\cH$-norm} \eqref{defHnorm}, we shall prove that 
\begin{equation}\label{claimmu21}
\begin{aligned}
|\partial_p^\alpha \Fmu^{(2)}| 
\lesssim 
\begin{cases}
	C_{\alpha}(t) 
	\zeta(t)^2  \langle k \rangle^{-\sigma-1} 
	\min\{ \langle k-p\rangle, \langle p\rangle\}^{-\sigma},\quad					&\textrm{if} \quad |\alpha| \le N-2, \\
	C_{\alpha}(t)  \zeta(t)^2  \langle k \rangle^{-\sigma} 
	\min\{ \langle k-p\rangle, \langle p\rangle\}^{-\sigma},  &\textrm{if} \quad |\alpha| \ge N-1.
\end{cases}
\end{aligned}
\end{equation}
Let us first focus on $\partial_p^\alpha \Fmu^{(2)}$ with $|\alpha|\le N-2$. Note that $\partial_p A_{k-p,\ell-k+p} \sim |\ell|$ and $|\partial_p^{\beta} A_{k-p,\ell-k+p}| \lesssim |\ell|$ for $|\beta| \ge 2$. Hence, $|\partial_p^{\alpha'} e^{isA_{k-p,\ell-k+p}} |\lesssim \langle \ell s\rangle^{|{\alpha'}|}$. Therefore, using the bootstrap assumptions for $|\alpha''|\le N-1$, see \eqref{defzeta} and \eqref{rhoXchange}, we obtain 
\begin{equation}\label{bdFmu21}
\begin{aligned}
|\partial_p^\alpha \Fmu^{(2)}|
&\lesssim \zeta(t)^2
\sum_{{\alpha'} + {\alpha''} \le \alpha}\int_0^t \int C_{\alpha''}(s)\langle s\rangle^\delta \langle \ell\rangle^{-\sigma}\langle \ell s\rangle^{-N+|{\alpha'}|} 
\\&\qquad \times |k-\ell|^{-\delta}\langle k-\ell\rangle^{-\sigma-1+\delta}\min\{\langle k-p-\ell\rangle, \langle p\rangle\}^{-\sigma} \, d\ell \, ds.
\end{aligned}\end{equation}
Note that 
\begin{equation}\label{goodtriangle}
	\langle \ell\rangle^{-\sigma} \langle k-\ell\rangle^{-\sigma-1} \le 2^\sigma \langle \ell\rangle \langle k\rangle^{-\sigma-1} \min\{ \langle \ell\rangle, \langle k-\ell \rangle\}^{-\sigma},
\end{equation}
which in particular yields the right decay in $k$ as in \eqref{claimmu21} (noting an additional growth factor in $\ell$). In addition, we claim that
\begin{equation}\label{claimineqkp}
\min\{ \langle \ell\rangle, \langle k-\ell \rangle\}^{-\sigma}\min\{\langle k-p-\ell\rangle, \langle p\rangle\}^{-\sigma} \lesssim \min\{\langle k-p\rangle, \langle p\rangle\}^{-\sigma}B(k,p,\ell),
\end{equation}
with 
\begin{equation}\label{defBkpl}B(k,p,\ell) = \min\{ \langle \ell\rangle, \langle k-\ell \rangle,\langle k-p-\ell\rangle\}^{-\sigma} \end{equation}
which would then yield the right decay in $p$ as required in \eqref{claimmu21}. Note that $\int B(k,p,\ell) \; d\ell \lesssim 1$ uniformly in $k,p$. Indeed, if $\min\{\langle k-p-\ell\rangle, \langle p\rangle\}^{-\sigma}  = \langle p\rangle^{-\sigma}$, then 
 the claim \eqref{claimineqkp} clearly holds. 
On the other hand, if $\min\{\langle k-p-\ell\rangle, \langle p\rangle\}^{-\sigma}  = \langle k-p-\ell\rangle^{-\sigma}$, we bound
$$ \begin{aligned}
\langle \ell \rangle^{-\sigma}\langle k-p-\ell\rangle^{-\sigma} &\le 2^\sigma \langle k-p\rangle^{-\sigma}\min\{ \langle \ell \rangle, \langle k-p-\ell\rangle\}^{-\sigma},
\\
\langle k-\ell \rangle^{-\sigma}\langle k-p-\ell\rangle^{-\sigma} &\le 2^\sigma \langle p\rangle^{-\sigma}\min\{ \langle k-\ell \rangle, \langle k-p-\ell\rangle\}^{-\sigma}.
\end{aligned}$$
This proves the claim \eqref{claimineqkp}. 
Therefore, returning to \eqref{bdFmu21} and using $\langle s\rangle^\delta \le \langle \ell s\rangle^\delta |\ell|^{-\delta} \langle \ell \rangle^{\delta}$, it remains to bound 
\begin{equation}\label{claimCa}
\begin{aligned}
\int_0^t \int C_{\alpha''}(s)  |k-\ell|^{-\delta} \langle k-\ell \rangle^{\delta}|\ell|^{-\delta}\langle \ell \rangle^{1+\delta} \langle \ell s\rangle^{-N+|{\alpha'}|+\delta}B(k,p,\ell)\; d\ell \, ds \lesssim C_\alpha(t),
\end{aligned}
\end{equation}
for ${\alpha'} + {\alpha''} \le \alpha$. First, consider the case when $|{\alpha'}|\le N-2$. We bound $C_{\alpha''}(s) \le C_\alpha(t)$, since $|\alpha''|\le |\alpha|$, and integrate in $s$, yielding 
\begin{align*}
&\int_0^t \int C_{\alpha''}(s)  |k-\ell|^{-\delta} \langle k-\ell \rangle^{\delta}|\ell|^{-\delta}\langle \ell \rangle^{1+\delta} \langle \ell s\rangle^{-2+\delta}B(k,p,\ell)\; d\ell \, ds \\
&\lesssim  C_\alpha(t)\int |k-\ell|^{-\delta} \langle k-\ell \rangle^{\delta}|\ell|^{-1-\delta}\langle \ell\rangle^{1+\delta} B(k,p,\ell)\; d\ell \\
&\lesssim C_\alpha(t),
\end{align*}
as claimed in \eqref{claimCa}, upon noting $\|B(k,p,\ell)\|_{L^1_\ell \cap L^\infty_\ell} \lesssim 1$, see \eqref{defBkpl}. 

Next, we consider the case when $|\alpha| \ge N-1$.
In view of the definition of $\cH^{|\alpha|}_*$ in \eqref{defzeta}, we only need to propagate a weaker decay of order $\langle k\rangle^{-\sigma}$, instead of $\langle k\rangle^{-\sigma-1}$ as in \eqref{goodtriangle}. As a result, we do not lose any decay in $\ell$ in this case (i.e. no prefactor $\langle \ell \rangle$ is needed in \eqref{goodtriangle}).
Therefore, cf. \eqref{claimCa}, it suffices to bound 
\begin{equation}\label{claimCaN}
	\begin{aligned}
		\int_0^t \int C_{\alpha''}(s) |k-\ell|^{-\delta} \langle k-\ell \rangle^{\delta} |\ell|^{-\delta} \langle \ell \rangle^{\delta}\langle \ell s\rangle^{-N+|{\alpha'}|+\delta}B(k,p,\ell)\; d\ell \, ds \lesssim C_\alpha(t),
	\end{aligned}
\end{equation}
for $\alpha'+\alpha'' \le \alpha$ and $|\alpha| \ge N-1$. The case when $|\alpha'|\le N-2$ is already treated above.
In the case when $|{\alpha'}| = N-1$, we have $C_{\alpha''}(s) =1$, since $|{\alpha''}|\le |\alpha| - |{\alpha'}|\le N - |{\alpha'}| = 1$. We also note that $\int_0^t \langle \ell s \rangle^{-1+\delta} \, ds \lesssim |\ell|^{-1} \langle \ell t \rangle^{\delta} \lesssim |\ell|^{-1} \langle \ell \rangle^{\delta} \langle t \rangle^{\delta}$. Therefore, we bound 
$$
\begin{aligned}
&\int_0^t \int C_{\alpha''}(s) |k-\ell|^{-\delta} \langle k-\ell \rangle^{\delta} |\ell|^{-\delta} \langle \ell \rangle^{\delta}\langle \ell s\rangle^{-1+\delta}B(k,p,\ell)\; d\ell \, ds\\
&\lesssim 
\langle t\rangle^\delta\int |k-\ell|^{-\delta} \langle k-\ell \rangle^{\delta} |\ell|^{-1-\delta} \langle \ell \rangle^{2\delta}\langle \ell s\rangle^{-1+\delta}B(k,p,\ell) \; d\ell \lesssim  \langle t\rangle^\delta, 
\end{aligned}$$
which is bounded by $C_\alpha(t)$, since $|\alpha|\ge |{\alpha'}| = N-1$. 
In the case when $|\alpha'| = |\alpha| = N$, we note that $C_{\alpha''}(s) =1$ (because ${\alpha''}=0$) and $\int_0^t \langle \ell s\rangle^\delta \; ds \le \langle \ell\rangle^\delta \langle t\rangle^{1+\delta}$. The claim \eqref{claimCaN} thus follows. This yields the bounds on $\Fmu^{(2)}$ as claimed in \eqref{claimmu21}. 

\subsubsection*{Bounds on $\Fmu^{(3)}$.}

Finally, we bound the nonlinear contribution $\Fmu^{(3)}$. By construction, we compute 
$$\partial_p^\alpha \Fmu^{(3)}
= 
\sum_{{\alpha'} + {\alpha''}+\alpha''' = \alpha}\int_0^t \int \partial_p^{\alpha'} e^{-is A_{p,\ell-p}} \partial_p^{\alpha''}\widehat{\mu}(s,k-p,p-\ell) \partial_p^{\alpha'''}\widehat{\cP}_{\gamma}(s, \ell-p, p) \; d\ell \, ds.
$$
Note that $\partial_p A_{p,\ell-p} \sim |\ell|$ and $|\partial_p^{\beta} A_{p,\ell-p}| \lesssim |\ell|$ for $|\beta| \ge 2$ so $|\partial_p^{\alpha'} e^{-is A_{p,\ell-p}} |\lesssim \langle \ell s\rangle^{|{\alpha'}|}$. In addition, 
the bootstrap assumptions \eqref{defzeta} now read 
$$
\begin{aligned}
&|\partial_p^{\alpha''}\widehat{\mu}(s,k-p,p-\ell)|\\ &\quad \le 
\begin{cases}
	C_{\alpha''}(s) \zeta(s) |k-\ell|^{-\delta} \langle k-\ell\rangle^{-\sigma-1 + \delta} \min\{ \langle k-p\rangle, \langle p-\ell\rangle\}^{-\sigma},\quad					&\textrm{if} \quad |\alpha''| \le N-2 , \\
	C_{\alpha''}(s) \zeta(s) |k-\ell|^{-\delta} \langle k-\ell\rangle^{-\sigma+ \delta} \min\{ \langle k-p\rangle, \langle p-\ell\rangle\}^{-\sigma},  &\textrm{if} \quad |\alpha''| \ge N-1 .
\end{cases}
\end{aligned}
$$
Following the above proof for $\Fmu^{(2)}$, it therefore suffices to obtain a similar bound as in \eqref{claimineqkp}, namely 
\begin{equation}\label{claimineqkp2}
\min\{ \langle \ell\rangle, \langle k-\ell \rangle\}^{-\sigma}\min\{ \langle k-p\rangle, \langle p-\ell\rangle\}^{-\sigma} \lesssim \min\{\langle k-p\rangle, \langle p\rangle\}^{-\sigma}B'(k,p,\ell)
\end{equation}
with $B'(k,p,\ell) = \min\{ \langle \ell\rangle, \langle k-\ell \rangle,\langle p-\ell\rangle\}^{-\sigma} $. Indeed, if $\min\{ \langle k-p\rangle, \langle p-\ell\rangle\} = \langle k-p\rangle$, then the inequality is direct. On the other hand, if $\min\{ \langle k-p\rangle, \langle p-\ell\rangle\} = \langle p-\ell\rangle$, then we bound 
$$ \begin{aligned}
\langle \ell \rangle^{-\sigma}\langle p-\ell\rangle^{-\sigma} &\le 2^\sigma \langle p\rangle^{-\sigma}\min\{ \langle \ell \rangle, \langle p-\ell\rangle\}^{-\sigma} ,
\\
\langle k-\ell \rangle^{-\sigma}\langle p-\ell\rangle^{-\sigma} &\le 2^\sigma \langle k-p\rangle^{-\sigma}\min\{ \langle k-\ell \rangle, \langle p-\ell\rangle\}^{-\sigma} ,
\end{aligned}$$
which yields \eqref{claimineqkp2}. Note that by construction, it remains valid that $\|B'(k,p,\ell)\|_{L^1_\ell \cap L^\infty_\ell}\lesssim 1$. Therefore, the desired bounds on  $\Fmu^{(3)}$ follow similarly to those on  $\Fmu^{(2)}$, which we skip to repeat the details. 

This completes the proof of Proposition \ref{prop-profile}.
\end{proof}

\subsection{Density estimates}\label{sec:nonlinear-density}

In this section, we establish decay estimates for the density. Precisely, we obtain the following key proposition. 

\begin{proposition}\label{prop-density}
	Let $\rho(t,x) = \gamma(t,x,x)$ be the density generated by the nonlinear Hartree--Fock equation \eqref{eqn:perturb}, and let 
$\widehat{\rho}(t, k)$ be its Fourier transform. Then, as long as $\zeta(t)$ remains finite, there hold 
$$
 \langle t\rangle^{-\delta} \| \langle k \rangle^{\sigma} \langle kt \rangle^{N} \widehat{\rho}(t,k) \|_{L^{\infty}_k}\le C_0 \epsilon_0 + C_1\epsilon_1 \zeta_\rho(t)+ C_2 \zeta(t)^2$$
for some universal constants $C_0,C_1,C_2$. 
\end{proposition} 

The proof of Proposition \ref{prop-density} uses the linear theory developed in the previous section, and the following representation of the nonlinear density.  

\begin{proposition}\label{prop:resolvent-nonlinear}
	Let 
$\Frho(t, k)$ be the Fourier transform of the Hartree--Fock density. Then, there holds 
\begin{equation}\label{repdensity}
\Frho(t,k) = \widehat{S}(t,k) + \int_0^t \FG^r(t-s, k) \widehat{S}(s,k)  \, ds.
\end{equation}
where $\FG^r(t, k)$ denotes the Green function constructed as in Proposition \ref{prop:Green}, and the source term $\FS(t,k)$ is computed by 
\begin{equation}\label{defSkt}
	\FS(t,k) = \int e^{-itA_{k-p,p}} \hgamma_0 (k-p,p) \, dp -i  \int_0^t \int e^{-i(t-s)A_{k-p,p}} \FF(s,k-p,p) \, dp \, ds,
\end{equation}
where $\hgamma_0(k,p)$ is the Hartree--Fock initial data, and 
\begin{equation}\label{defFF}	
\begin{aligned}
		\FF(t,k,p) &= a_{k,p} \widehat{\mathcal{X}}_{\gamma}(t,k,p) 
		&+ \frac{1}{(2\pi)^d}\int \left(\widehat{\mathcal{P}}_{\gamma}(t,k,\eta) \widehat{\gamma}(t,-\eta, p) -  \widehat{\gamma}(t,k,\eta) \widehat{\mathcal{P}}_{\gamma}(t,-\eta, p) \right)  d\eta.
	\end{aligned}\end{equation}
\end{proposition}
\begin{proof} Indeed, it follows from \eqref{eqn:fourier}, now evaluated at $(k-p,p)$, that 
\[
	(i\partial_t - A_{k-p,p} )\Fgamma(t,k-p,p) = - a_{k-p,p} \hw(k) \Frho_{\gamma}(t,k) + \FF(t,k-p,p),
\]
with $\FF(t,k,p)$ defined as in \eqref{defFF}. Applying the linear theory developed in the previous section yields the representation as stated. 
\end{proof}

\begin{lemma}\label{lem-reduceS}
Let $\Frho(t,k)$ satisfy \eqref{repdensity}. Then, as long as $\zeta(t)$ remains finite, there holds 
\[
\| \langle k \rangle^{\sigma} \langle kt \rangle^{N} \widehat{\rho}_\gamma(t,k) \|_{L^{\infty}_k}\lesssim \sup_{s\in [0,t]} \| \langle k \rangle^{\sigma} \langle ks \rangle^{N} \widehat{S}(s,k) \|_{L^{\infty}_k},
\]
for any $\sigma\ge 0$. 
\end{lemma}
\begin{proof} In view of the representation \eqref{repdensity}, we only need to bound the time integral term that involves the Green function $\FG^r(t,k)$. Indeed, using Proposition \ref{prop:Green}, precisely $|\widehat{G}^r(t,k)  | \lesssim |k| \langle kt \rangle^{-n_0 +3}$, we bound 
\begin{equation*}
	\begin{split}
\left|  \int_{0}^{t} \widehat{G}^{r}(t-s, k) \widehat{S}(s,k)\, ds \right|
		&\lesssim |k| \int_{0}^{t} \langle k(t-s) \rangle^{-n_0 +3} |\widehat{S}(s,k)|\;ds \\
		&\lesssim  \sup_{s\in [0,t]} \| \langle k \rangle^{\sigma} \langle ks \rangle^{N} \widehat{S}(s,k) \|_{L^{\infty}_k} |k| \int_{0}^{t} \langle k(t-s) \rangle^{-n_0 +3} \langle k\rangle^{-\sigma} \langle ks\rangle^{-N}\;ds
		\\
		&\lesssim  \sup_{s\in [0,t]} \| \langle k \rangle^{\sigma} \langle ks \rangle^{N} \widehat{S}(s,k)\|_{L^{\infty}_k} \langle k\rangle^{-\sigma} \langle kt\rangle^{-N}, 
	\end{split}
\end{equation*}
provided $n_0 -3 \ge N$. The lemma follows. 
\end{proof}

\begin{proof}[Proof of Proposition \ref{prop-density}]
We now prove Proposition \ref{prop-density}. Thanks to Lemma \ref{lem-reduceS}, it remains to bound the source term $\FS(t,k)$. By definition of $\FS(t,k)$ in \eqref{defSkt}, 
we write 
\[
	\FS(t,k) = \Frho^{0}(t,k) + \Frho^{X}(t,k) + \Frho^{\Phi}(t,k) + \Frho^{\Psi}(t,k),
\]
where we have set 
\begin{equation}\label{defrho}	
\begin{aligned}
	\Frho^{0}(t,k)&= \int e^{-itA_{k-p,p}} \hgamma_0 (k-p,p) \, dp, \\
	\Frho^{X}(t,k) &= -i\int_0^t \int e^{-i(t-s)A_{k-p,p}} a_{k-p,p} \widehat{\mathcal{X}}_{\gamma}(s,k-p,p)  \, dp \, ds, \\
	\Frho^{\Phi}(t,k) &=- \frac{i}{(2\pi)^d} \int_0^t \iint e^{-i\Phi} \widehat{\cP}_\gamma (s,k-p,\ell-k+p)\widehat{\mu}(s,k-p-\ell,p) \, d\ell \, dp \, ds,  \\
	\Frho^{\Psi}(t,k)
	&= \frac{i}{(2\pi)^d} \int_0^t \iint
	e^{-i\Psi}\, \widehat{\cP}_\gamma (s,-p+\ell,p) \widehat{\mu}(s,k-p,p-\ell)\, d\ell \, dp\, ds,
	\end{aligned}\end{equation}
with the phase functions 
\begin{equation}\label{defPhi}
\begin{aligned}
\Phi &= tA_{k-p,p} - sA_{k-p,\ell-k+p},
\\
\Psi&= tA_{k-p,p} - sA_{p,\ell-p}.
\end{aligned}
\end{equation} 
We stress that the phase functions depend on all the Fourier variables. 
Let us now bound each term. First, the contribution from the initial data is direct, upon using the oscillation $e^{-itA_{k-p,p}}$. Indeed, using Lemma \ref{lemAphase}, we obtain  
$$
\begin{aligned}
	|\Frho^{0}(t,k)|
&	\lesssim \langle kt \rangle^{-N} \sum_{|\alpha| \le N} \int |\partial_p^{\alpha} \widehat{\gamma_0} (k-p,p) | \, dp
	\lesssim \epsilon_0\langle kt \rangle^{-N} \int \langle k-p\rangle^{-\sigma_0}\langle p\rangle^{-\sigma_0} \, dp
,
\end{aligned}
$$
which is bounded by $C_0\epsilon_0 \langle kt \rangle^{-N} \langle k \rangle^{-\sigma_0}$, provided $\sigma_0>d$ (upon using \eqref{minab}). This proves the desired estimate on $\Frho^{0}(t,k)$, provided $\sigma_0\ge \sigma$.

\subsubsection{Linear reaction $\rho^X$}

In this section, we bound the linear contribution associated to the exchange term, namely the density $\rho^X$. Making use of the oscillation $e^{-i(t-s)A_{k-p,p}} $ and using again Lemma \ref{lemAphase}, we bound 
\begin{align*}
	|\Frho^{X}(t,k)|
	&\lesssim  \sum_{|\alpha_1| + |\alpha_2| \le N} \int_0^t \langle k (t-s) \rangle^{-N}\int \left| \partial_p^{\alpha_1} a_{k-p,p}\,  \partial_p^{\alpha_2} \widehat{\cX}_\gamma (s,k-p,p) \right| \, dp \, ds. 
	\end{align*}
Recall from Lemma \ref{lem:akp} that $|k|^{-1}\partial_p^{\alpha_1} a_{k-p,p}$ decays rapidly in $\min\{\langle k-p\rangle,\langle p\rangle\}$, and the integrability in $p$ is thus valid. Therefore, using the bootstrap assumptions in \eqref{defzeta}, we bound 
\begin{align*}
	|\Frho^{X}(t,k)|
	&\lesssim \epsilon_1\| |k|^{-1}a_{k-p,p} \|_{W^{N, 1}_p} \int_0^t \langle k(t-s) \rangle^{-N} \langle ks \rangle^{-N} \langle k \rangle^{-\sigma} \langle s\rangle^{\delta}\zeta_\rho(s) \, |k|\, ds \\
	&\lesssim \epsilon_1\langle kt \rangle^{-N}  \langle k \rangle^{-\sigma}\langle t\rangle^{\delta} \zeta_\rho(t),
\end{align*}
as desired, provided that $N\le n_0 -1$.

\subsubsection{Nonlinear reaction $\rho^{\Phi,\Psi}$}

Next, we treat the nonlinear contribution, namely $\Frho^\Phi$ and $\Frho^\Psi$. We shall focus precisely on bounding $\Frho^\Phi$, namely 
\begin{equation}\label{defrhoH1}
\begin{aligned}	
	\Frho^{\Phi}(t,k) &=- \frac{i}{(2\pi)^d} \int_0^t \int e^{-i\Phi} \widehat{\cP}_\gamma (s,k-p,\ell-k+p)\widehat{\mu}(s,k-p-\ell,p) \, d\ell \, dp \, ds,
\end{aligned}
\end{equation}
where the phase function $\Phi$ is defined as in \eqref{defPhi}. The estimates on $\Frho^\Psi$ are nearly identical, and we shall therefore skip to repeat the details. Precisely, we will prove that 
\begin{equation}\label{bdrhoH}
|\widehat{\rho}^\Phi(t,k)|\le C \zeta(t)^2 \langle t\rangle^{\delta}  \langle k \rangle^{-\sigma} \langle kt \rangle^{-N} 
 \end{equation}
uniformly in $k$ and $t\ge 0$, for some universal constant $C$. As in the previous case, we shall make a crucial use of oscillation $e^{i\Phi}$ to derive decay in $kt$. Unfortunately, Lemma \ref{lem:IBP} is not applicable directly, since the ratio of $\partial_p^\alpha \Phi$ to $|\partial_p \Phi|$ may not be universally bounded, due to resonances (i.e. when $\partial_p \Phi$ is a suborder of $|kt|$, while higher derivatives $\partial_p^\alpha\Phi$ is of order $|kt|$ due to the perturbation $m(k)$ in the dispersion relation $\omega(k)$). To overcome this issue, we are obliged to further examine the resonance region 
when $\{ |\partial_p\Phi|\ll |kt|\}$. Precisely, let 
\begin{equation}\label{deftheta0}
	\theta_0 : = \inf_{k,p} \frac{|\partial_p A_{k-p,p}|}{|k|}.
\end{equation}
In view of Lemma \ref{lem:Akp}, $\theta_0 \ge 2- \| \nabla^2 m \|_{L^{\infty}}$, which is well-defined and strictly positive, since $\| \nabla^2 m \|_{L^{\infty}}$ is sufficiently small.
We take a sufficiently smooth cut-off function $\chi(x)$ which vanishes for $|x|\le 1$ and is equal to one for $|x|\ge 2$, and introduce 
\begin{equation}\label{cutoff} 
\varphi_\Phi = \chi \Big(\frac{4\partial_p\Phi}{\theta_0 |kt|}\Big)
\end{equation}
with $\theta_0$ as in \eqref{deftheta0}. Going back to the integral $\Frho^{\Phi}(t,k)$ in \eqref{defrhoH1}, we integrate by parts in $p$ once and put appropriate cut-off functions to distinguish the resonant and non-resonant regions. Indeed, upon writing $e^{-i\Phi} = L_pe^{-i\Phi}$, with  
\begin{equation}\label{recallL}
		L_p f = \frac{1+ i\nabla_p \Phi \cdot \nabla_p f}{1+|\nabla_p \Phi|^2}, \qquad L^{\ast}_p g = \frac{g}{1+|\nabla_p \Phi|^2} -i \nabla_p \cdot \left( \frac{\nabla_p \Phi}{1+|\nabla_p \Phi|^2} g \right),
	\end{equation}
we write 
\begin{equation}\label{rhoHRes}
\Frho^{\Phi}(t,k) = \Frho^{\Phi, R}(t,k) + \Frho^{\Phi, NR}(t,k),
\end{equation}
where 
\begin{equation}\label{redefrhoH}
\begin{aligned}
\Frho^{\Phi,R}(t,k)&= \int_0^t \int e^{-i\Phi} L^{\ast}_p \Big( \widehat{\cP}_\gamma (s,k-p,\ell-k+p)\widehat{\mu}(s,k-p-\ell,p)\Big) (1-\varphi_\Phi)\, dp \, d\ell \, ds,
\\
\Frho^{\Phi,NR}(t,k)&= \int_0^t \int e^{-i\Phi} L^{\ast}_p \Big( \widehat{\cP}_\gamma (s,k-p,\ell-k+p)\widehat{\mu}(s,k-p-\ell,p)\Big) \varphi_\Phi \, dp \, d\ell \, ds,
\end{aligned}
\end{equation}
which correspond to resonant and non-resonant regions. We treat each case separately. 

\subsubsection*{Non-resonant Case: $|\partial_p \Phi|\ge \theta_0 |kt|/4$.}

We shall argue that in this non resonant case, Lemma \ref{lem:IBP} can be applied. Indeed, 
in this case, we note that by construction, the support of $\varphi_\Phi$ implies that $|\partial_p \Phi| \ge \theta_0 |kt|/4$. This, together with the fact from \eqref{dispAkp} that $\partial_p A_{k-p,p} \sim |k|$ and $\partial_pA_{k-p,\ell-k+p} \sim |\ell|$, we have 
\begin{equation}\label{bdls} 
|\ell s | \lesssim  |s\partial_pA_{k-p,\ell-k+p} | \le |t\partial_p A_{k-p,p}| + |\partial_p \Phi| \lesssim |\partial_p \Phi|,  
\end{equation}
uniformly in $k,p,\ell$ and $s,t\ge 0$. On the other hand, recalling $A_{k,p} = \omega(k) - \omega(p)$ with $\omega(k) = |k|^2 + m(k)$ for sufficiently smooth and symmetric symbol $m(k)$, for $2 \le |\alpha| \le 3n_0 -1$, we compute 
$$ |\partial_p^\alpha A_{k-p,p}| \lesssim |\partial_p^\alpha m(p-k) - \partial_p^\alpha m(p)| \lesssim |k|$$
which yields 
\begin{equation}\label{bddPa} |\partial_p^\alpha \Phi| \le |t\partial_p^\alpha A_{k-p,p}| + |s\partial_p^\alpha A_{k-p,\ell-k+p}|\lesssim |kt| + |\ell s| \lesssim |\partial_p\Phi|,
\end{equation}
for $2\le |\alpha| \le 3n_0 -1$, upon using \eqref{bdls}. That is, the non-resonant condition \eqref{phase-cond} is satisfied, and Lemma \ref{lem:IBP} can be applied in this case. In addition, we note that derivatives of $\chi$ are supported in $1 \le |x| \le 2$, so that $|\partial_p \Phi| \sim |kt|$ on the support of $\partial_p^{\alpha} \varphi_{\Phi}$. Hence, the cut-off functions $\varphi_\Phi$ are regular with bounded derivatives. Therefore, integrating by parts and using the bootstrap assumptions in \eqref{defzeta} and \eqref{rhoXchange}, we bound  
\begin{equation}\label{bdrhoH0}
\begin{aligned}	
|\Frho^{\Phi,NR}(t,k)| 
&\lesssim \zeta(t)^2\int_0^t \int_{\{|\partial_p \Phi|\gtrsim |kt|\}} \langle s\rangle^{1+2\delta}\langle \ell\rangle^{-\sigma}\langle \ell s\rangle^{-N} \langle \partial_p\Phi\rangle^{-N} 
\\&\qquad \times |k-\ell|^{-\delta}\langle k-\ell \rangle^{-\sigma+\delta}\min\{\langle k-\ell-p\rangle, \langle p\rangle\}^{-\sigma}\, dp \, d\ell \, ds
\\
&\lesssim \zeta(t)^2\langle kt\rangle^{-N}\int_0^t \int \langle s\rangle^{1+2\delta}\langle \ell\rangle^{-\sigma}\langle \ell s\rangle^{-N} |k-\ell|^{-\delta}\langle k-\ell \rangle^{-\sigma+\delta} \, d\ell \, ds,
\end{aligned}
\end{equation}
upon using the non-resonant condition and the fact that $\min\{\langle k-\ell-p\rangle, \langle p\rangle\}^{-\sigma}$ is integrable in $p$. Next, using 
$$\langle k-\ell \rangle^{-\sigma}\langle \ell\rangle^{-\sigma} \le 2^{\sigma}\langle k\rangle^{-\sigma}
 \min\{ \langle k-\ell\rangle, \langle \ell\rangle\}^{-\sigma},$$
which yields the desired decay in $k$. Therefore, it suffices to prove the following bound
\begin{equation}\label{claim:NR}
	\int \langle \ell s \rangle^{-N} |k-\ell|^{-\delta} \langle k -\ell \rangle^{\delta} \min \{ \langle k- \ell \rangle, \langle \ell \rangle \}^{-\sigma} \, d\ell
	\lesssim \langle s \rangle^{-d+2\delta}
\end{equation}
uniformly in $k$. Indeed, if \eqref{claim:NR} holds, then we have
\begin{equation}\label{bdrhoH1}
\begin{aligned}	
|\Frho^{\Phi,NR}(t,k)| 
&\lesssim \zeta(t)^2\langle kt\rangle^{-N}\langle k\rangle^{-\sigma}\int_0^t \int \langle s\rangle^{1+2\delta}\langle \ell s \rangle^{-N} |k-\ell|^{-\delta} \langle k -\ell \rangle^{\delta} \min \{ \langle k- \ell \rangle, \langle \ell \rangle \}^{-\sigma}\, d\ell \, ds
 \\
 &\lesssim \zeta(t)^2\langle kt\rangle^{-N}\langle k\rangle^{-\sigma}\int_0^t \langle s\rangle^{-d+1+4\delta}\; ds
  \\
 &\lesssim \zeta(t)^2\langle kt\rangle^{-N}\langle k\rangle^{-\sigma},
\end{aligned}
\end{equation}
as desired, for sufficiently small $\delta>0$. We now prove the claim \eqref{claim:NR}.
The case when $|k-\ell| \ge 1$ is direct because we have $|k-\ell|^{-\delta} \langle k-\ell \rangle^{\delta} \lesssim 1$, so that we bound
\[
	\int_{\{|k-\ell| \ge 1 \}} \langle \ell s \rangle^{-N} \min \{ \langle k- \ell  \rangle, \langle \ell \rangle\}^{-\sigma} \, d\ell \lesssim \langle s \rangle^{-d}.
\]
When $|k-\ell| \le 1$, we have $\langle k-\ell \rangle^{\delta} \lesssim 1$. Using the H\"older inequality, we bound
\begin{equation*}
\begin{aligned}
	\int_{\{|k-\ell| \le 1\}} \langle \ell s \rangle^{-N} |k-\ell|^{-\delta} \, d\ell
	&\lesssim \left(\int_{\{|k-\ell| \le 1\}} \langle \ell s \rangle^{-Np} \, d\ell \right)^{1/p}
		\left(\int_{\{|k-\ell| \le 1\}}  |k-\ell|^{-\delta p'} \, d\ell \right)^{1/p'} \\
	&\lesssim \langle s \rangle^{-d/p}
	\lesssim \langle s \rangle^{-d+2\delta},
\end{aligned}
\end{equation*}
where $p= d/(d-2\delta)$ and $p'= d/(2\delta)$, so that $1/p + 1/p' = 1$, and $|k-\ell|^{-\delta p'}$ is locally integrable in $\ell$. This proves \eqref{claim:NR}.

\subsubsection*{Resonant Case: $|\partial_p \Phi| \le \theta_0 |kt|/2$.}

Next, we consider the resonant integral $\Frho^{\Phi,R}(t,k)$ defined as in \eqref{redefrhoH}. By construction, the support of $1-\varphi_\Phi$ is contained in the resonant region $\{|\partial_p \Phi| \le \theta_0 |kt|/2\}$ with $\theta_0$ as in \eqref{deftheta0}. 
Recall that $\partial_p A_{k-p,p} \sim |k|$ and $\partial_pA_{k-p,\ell-k+p} \sim |\ell|$. Therefore, in this resonant case, it follows that 
\begin{equation}\label{resbdls} 
	|\ell s | \gtrsim |s\partial_pA_{k-p,\ell-k+p} | \ge |t\partial_p A_{k-p,p}| - |\partial_p \Phi| \ge \theta_0 |kt| - |\partial_p \Phi| \ge \theta_0|kt| /2,  
\end{equation}
uniformly in $k,p,\ell$ and $s,t\ge 0$.
As a result, the decay of order $\langle \ell s\rangle^{-N}$ due to the bootstrap norm of $\Frho(s,\ell)$ yields the desired decay of order $\langle kt\rangle^{-N}$. Therefore, we bound 
$$
\begin{aligned}	
|\Frho^{\Phi, R}(t,k)| &\lesssim \int_0^t \int_{\{|\partial_p \Phi| \le \theta_0  |kt|/2\}}  \Big[ \frac{1}{\langle \partial_p\Phi\rangle} + \frac{|\partial_p^2\Phi|}{\langle \partial_p\Phi\rangle^2}\Big] 
\\&\qquad \times |\langle \partial_p\rangle  \widehat{\cP}_\gamma (s,k-p,\ell-k+p)||\langle \partial_p\rangle\widehat{\mu}(s,k-p-\ell,p)|\, dp \, d\ell \, ds.
\end{aligned}
$$
Recall that $A_{k-p,p} = \omega(p-k) - \omega(p)$, with $\omega(p) = |p|^2 + m(p)$ for some bounded $C^N$ function $m(p)$. As a result,  for $|\alpha|\ge 2$ and $0\le s\le t$, we compute 
$$
|\partial_p^\alpha \Phi| \le |t\partial_p^\alpha A_{k-p,p}| + |s\partial_p^\alpha A_{k-p,\ell-k+p}|\lesssim |t| \lesssim |kt| |k|^{-1} \lesssim |\ell s| |k|^{-1},
$$
upon using \eqref{resbdls}. In the case when $|k|\le 1$, we instead bound 
$$|\partial_p^\alpha \Phi| \le |t\partial_p^\alpha A_{k-p,p}| + |s\partial_p^\alpha A_{k-p,\ell-k+p}|\lesssim |kt| + |\ell s| \lesssim |\ell s|.
$$
Combining, for $|\alpha|\ge 2$, we obtain 
$$
|\partial_p^\alpha \Phi| \le |t\partial_p^\alpha A_{k-p,p}| + |s\partial_p^\alpha A_{k-p,\ell-k+p}|
\lesssim |\ell s| \langle k\rangle^{-1},
$$
uniformly in $k,p,\ell$, which yields  
\begin{equation}\label{resbddPa}
\frac{1}{\langle \partial_p\Phi\rangle} + \frac{|\partial_p^2\Phi|}{\langle \partial_p\Phi\rangle^2}
\lesssim \frac{1}{\langle \partial_p \Phi \rangle} + \frac{|\ell s | \langle k \rangle^{-1}}{\langle \partial_p \Phi \rangle^2}.
\end{equation}
Hence, together with the bootstrap assumptions in \eqref{defzeta} and \eqref{rhoXchange}, we bound 
\begin{equation}\label{rhoH1R}
	\begin{aligned}	
		|\Frho^{\Phi, R}(t,k)| &\lesssim \zeta(t)^2\int_0^t \int\langle \ell \rangle^{-\sigma}\langle \ell s\rangle^{-N} \langle s\rangle^\delta
		\left[ \frac{1}{\langle \partial_p \Phi \rangle} + \frac{|\ell s | \langle k \rangle^{-1}}{\langle \partial_p \Phi \rangle^2} \right]
		\\&\qquad \times |k-\ell|^{-\delta}\langle k-\ell \rangle^{-\sigma-1+\delta}\min\{\langle k-\ell-p\rangle, \langle p\rangle\}^{-\sigma}
		\, dp \, d\ell \, ds.
	\end{aligned}
\end{equation}
Here in the above estimate, we have used precisely the boostrap assumptions on $\langle \partial_p\rangle\widehat{\mu}$ (i.e. up to the first derivatives), and therefore no growth factor $t^{1+\delta}$ appears, cf. \eqref{bdrhoH0}, while gaining an extra factor of decay in $\langle k-\ell\rangle$. 

In this resonant case, $\partial_p\Phi$ is relatively small (when compared with $|kt|$), which requires delicate analysis to treat these echo resonances. In particular, $\partial_p\Phi$ depends on all the Fourier variables $k,p,\ell$, and therefore there are no apparent separation between $kt$ and $\ell s$, leading to a complication in taking integration (i.e. one cannot simply taking $p$ or $\ell$-integration separately in addition to the fact that the Jacobian determinant of the map $(p,\ell) \mapsto (p, \partial_p \Phi)$ may not be invertible, since $\partial_p\partial_\ell\Phi \sim |kt| + |\ell s|$).

To proceed, first using $\langle k\rangle^{-1}\langle k-\ell\rangle^{-1} \le 2 \langle \ell\rangle^{-1}$, we tame the growth factor of $|\ell|$ (that appears in $|\ell s|$). Next, using $\langle \ell s\rangle^{-N}\lesssim \langle kt\rangle^{-N}$, thanks to \eqref{resbdls}, 
and 
$$\langle k-\ell\rangle^{-\sigma}\langle \ell\rangle^{-\sigma} \lesssim \langle k\rangle^{-\sigma}
 \min\{ \langle k-\ell\rangle, \langle \ell\rangle\}^{-\sigma},$$
we obtain the desired decay in $k$ and $kt$ for $|\Frho^{\Phi, R}(t,k)|$. Therefore, it suffices to prove the following bound
 \begin{equation}\label{claimrhoH1R}
\begin{aligned}	
\FI(s,k) :&= \int	\left[ \frac{1}{\langle \partial_p \Phi \rangle} + \frac{| s | }{\langle \partial_p \Phi \rangle^2} \right] |k-\ell|^{-\delta} \langle k-\ell \rangle^{ \delta}\min\{ \langle k-\ell\rangle, \langle \ell\rangle\}^{-\sigma}\min\{ \langle k-\ell-p\rangle, \langle p\rangle\}^{-\sigma}
\, dp \, d\ell  \\
&\lesssim \langle s\rangle^{-1},
\end{aligned}
\end{equation}
uniformly in $k$. Indeed, if \eqref{claimrhoH1R} holds, we then have 
\begin{equation}\label{rhoH1Rest}
\begin{aligned}	
|\Frho^{\Phi, R}(t,k)| 
&\lesssim \zeta(t)^2\langle k \rangle^{-\sigma}\langle kt\rangle^{-N}  \int_0^t 
\langle s\rangle^\delta
 \FI(s,k)
\, ds
\\
&\lesssim \zeta(t)^2\langle k \rangle^{-\sigma}\langle kt\rangle^{-N}  \int_0^t 
\langle s\rangle^{-1+\delta}
\, ds
\\
&\lesssim \zeta(t)^2\langle k \rangle^{-\sigma}\langle kt\rangle^{-N} 
\langle t\rangle^{\delta}
\end{aligned}
\end{equation}
 as desired. We now focus on proving the claim \eqref{claimrhoH1R}. The case when $|s|\le 1$ is direct, since the functions $|k-\ell|^{-\delta} \langle k-\ell \rangle^{\delta}\min\{ \langle k-\ell\rangle, \langle \ell\rangle\}^{-\sigma}$ and $\min\{ \langle k-\ell-p\rangle, \langle p\rangle\}^{-\sigma}$ are integrable in $\ell$ and $p$. Next, we focus on the case when $s\ge 1$. We further consider two cases: $|k-\ell - p|\ge |p|/2$ and $|k-\ell - p|\le |p|/2$. 
 
 \subsubsection*{Resonant sub-Case I: $|k-\ell - p|\ge |p|/2$.}
 
 We first prove the claim \eqref{claimrhoH1R} in this subcase. First, since $|k-\ell - p|\gtrsim |p|$, we note that 
$$\min\{ \langle k-\ell-p\rangle, \langle p\rangle\}^{-\sigma} \lesssim \langle p\rangle^{-\sigma}.$$ 
 Therefore, by first integrating the integral $\FI(s,k)$ in $\ell$ and then in $p$, the claim \eqref{claimrhoH1R} would follow from the following bound
\begin{equation}\label{bdresPhi1}
\begin{aligned}
\int \left[ \frac{1}{\langle \partial_p \Phi \rangle} + \frac{| s | }{\langle \partial_p \Phi \rangle^2} \right] |k-\ell|^{-\delta} \langle k-\ell \rangle^{\delta} \min\{ \langle k-\ell\rangle, \langle \ell\rangle\}^{-\sigma}\, d\ell &\lesssim \langle s\rangle^{-1},
\end{aligned}
\end{equation}
uniformly in $k,p$ and $s\ge 1$. To prove \eqref{bdresPhi1}, we note that 
$$\langle \partial_p\Phi \rangle \ge |t\partial_pA_{k-p,p} - s\partial_pA_{k-p,\ell-k+p}| \ge s\Big| \partial_pA_{k-p,\ell-k+p} - \frac{t}{s}\partial_pA_{k-p,p} \Big|.$$ 
Recall that the map $\ell\mapsto \eta =\partial_pA_{k-p,\ell-k+p}$ is diffeomorphism and has Jacobian of order $\partial_\ell \eta \sim \partial_p^2 \omega(\ell-k+p) \sim 2$, upon recalling that $\omega(p) = |p|^2 + m(p)$, where $m(p)$ has a sufficiently small $C^2$ norm. In addition, $|\ell|\lesssim |\eta|\lesssim |\ell|$, uniformly in $k,p$. Therefore, provided that $\sigma>d$, we bound 
\begin{equation}\label{bdresPhi}
\begin{aligned}
&\int \left[ \frac{1}{\langle \partial_p \Phi \rangle} + \frac{| s | }{\langle \partial_p \Phi \rangle^2} \right] |k-\ell|^{-\delta} \langle k-\ell \rangle^{\delta} \min\{ \langle k-\ell\rangle, \langle \ell\rangle\}^{-\sigma}\, d\ell \\
&\lesssim |s|^{-1} \sum_{m=1}^2 \int \Big|\eta  - \frac{t}{s} \partial_pA_{k-p,p}\Big|^{-m} |k-\ell'|^{-\delta} \langle k-\ell' \rangle^{\delta}\min\{ \langle k-\ell'\rangle, \langle \ell'\rangle\}^{-\sigma}\, d\eta
\end{aligned}
\end{equation}
in which $\ell' = \ell'_{k,p}(\eta)$ denotes the inverse map of $\ell \mapsto \eta=\partial_pA_{k-p,\ell-k+p}$ for each $k,p$. Note that the function $|\eta|^{-2}$ is locally integrable, since $\eta \in \RR^d$ with $d\ge 3$. Therefore, setting $B = \frac{t}{s} \partial_pA_{k-p,p}$ which is independent of $\eta$, we bound for $m=1,2$ that
$$ 
\begin{aligned}
&\int |\eta  - B|^{-m} |k-\ell'|^{-\delta} \langle k-\ell' \rangle^{\delta}\min\{ \langle k-\ell'\rangle, \langle \ell'\rangle\}^{-\sigma} \, d\eta
\lesssim 1,
\end{aligned}
 $$
uniformly in $k,p$ and $B$. 
This proves \eqref{bdresPhi1}, and therefore the claim \eqref{claimrhoH1R} in this resonant sub-case. 

 \subsubsection*{Resonant sub-Case II: $|k-\ell - p|\le |p|/2$.}

Next, we prove the claim \eqref{claimrhoH1R} in the resonant case when $|k-\ell - p|\le |p|/2$. In this case, we note that  
$$ |k-\ell| \ge |p| - |k-\ell-p| \ge |p|/2.$$
Therefore, if $|k-\ell|\le |\ell|$, then 
$$
\min\{\langle k-\ell\rangle, \langle \ell\rangle\} =\langle k-\ell \rangle \gtrsim \langle p\rangle. 
$$ 
On the other hand, if $|k-\ell|\ge |\ell|$, then 
$$
\min\{\langle k-\ell\rangle, \langle \ell\rangle\} =\langle \ell \rangle \gtrsim \langle k-p\rangle \langle k-\ell-p\rangle^{-1}, 
$$ 
in which we have used $\langle a\rangle \le 2 \langle a-b\rangle \langle b\rangle$. As a result, we obtain 
\begin{equation}
\min\{\langle k-\ell\rangle, \langle \ell\rangle\}^{-\sigma} \langle k-\ell-p\rangle^{-\sigma} \lesssim \min\{\langle k-\ell\rangle, \langle \ell\rangle\}^{-\sigma/2} \min\{ \langle k-p\rangle, \langle p\rangle\}^{-\sigma/2}.
\end{equation}
Therefore, in this resonant sub-case, we bound 
$$
\begin{aligned}
& \int	\left[ \frac{1}{\langle \partial_p \Phi \rangle} + \frac{| s | }{\langle \partial_p \Phi \rangle^2} \right] |k-\ell|^{-\delta} \langle k-\ell \rangle^{ \delta}\min\{ \langle k-\ell\rangle, \langle \ell\rangle\}^{-\sigma}\min\{ \langle k-\ell-p\rangle, \langle p\rangle\}^{-\sigma}
\, dp \, d\ell 
\\
&\lesssim \int \min\{ \langle k-p\rangle, \langle p\rangle\}^{-\sigma/2} \left( \int \left[ \frac{1}{\langle \partial_p \Phi \rangle} + \frac{| s | }{\langle \partial_p \Phi \rangle^2} \right] |k-\ell|^{-\delta} \langle k-\ell \rangle^{ \delta} \min\{ \langle k-\ell\rangle, \langle \ell\rangle\}^{-\sigma/2}
\, d\ell\right)  dp.
\end{aligned}$$
Using the bound \eqref{bdresPhi} with $\sigma$ replaced by $\sigma/2$, provided $\sigma>2d$, we obtain that the above integral is bounded by $C_0 \langle s\rangle^{-1}$ as claimed in \eqref{claimrhoH1R}. 

This completes the proof of Proposition \ref{prop-density}. 
\end{proof}

\subsection{Exchange kernels}\label{sec:nonlinear-exchange}

In this section, we prove the boostrap bounds on the exchange term. Precisely, we obtain the following. 

\begin{proposition}\label{prop-exchange}
As long as $\zeta(t)$ remains finite, for any $|\alpha|\le N$, there hold 
\begin{equation}\label{exchangebd}
\langle t\rangle^{-\delta}\| \langle k \rangle^{\sigma} \langle kt \rangle^{N} \partial_q^{\alpha} \widehat{\cX}_\gamma(t,k-q,q) \|_{L^{\infty}_{k,q}} \le \epsilon_1\Big(C_0 \epsilon_0 + C_1\epsilon_1 \zeta_\rho(t)+ C_2 \zeta(t)^2\Big)\end{equation}
for some universal constants $C_0,C_1,C_2$. 
\end{proposition} 

\begin{proof} Recalling \eqref{compcXkp}, we compute 
\begin{align*}
 \widehat{\cX}_\gamma(t,k-q,q) 
 &=  \int \hw_2(q-p) e^{-it A_{k-p,p}}\widehat{\mu}(t,k-p,p) \, dp.
\end{align*}
In particular, we note that $q$-derivatives now fall precisely only on $\hw_2(q-p)$. Namely, for any $|\alpha|\le N$, we have 
\[
	\partial_q^{\alpha} \widehat{\cX}_\gamma(t,k-q,q) = \int \partial_q^{\alpha} \hw_2(q-p) e^{-it A_{k-p,p}} \widehat{\mu}(t,k-p, p) \, dp.
\]
Observe that, up to the presence of $\partial_q^{\alpha} \hw_2(q-p)
$, the above integral is identical to the density $\Frho(t,k) = \int e^{-it A_{k-p,p}} \widehat{\mu}(t,k-p, p) \, dp$, and therefore, we may follow similar calculations as done in the previous section to establish the desired decay for $\partial_q^{\alpha} \widehat{\cX}_\gamma(t,k-q,q)$, provided that $|\partial^\alpha\hw_2(p)| \lesssim \epsilon_1$ for $|\alpha|\le 2N$.
Specifically, for sake of completeness, using \eqref{eqn:profile}, 
we compute 
\begin{align*}
		\partial_q^{\alpha} \widehat{\cX}_{\gamma}(t,k-q,q) 
		&= \cI + \cL^X + \cN^\Phi + \cN^\Psi
\end{align*}
where
\begin{align*}
	\cI &= \int \partial^{\alpha} \hw_2(q-p) e^{-itA_{k-p,p}} \Fgamma_0 (k-p,p) \, dp, \\
	\cL^X &= i \int_0^t \int \partial^{\alpha} \hw_2(q-p) e^{-i(t-s) A_{k-p, p}} a_{k-p, p} \widehat{\cP}_{\gamma}(s,k-p, p) \, dp \, ds, \\
	\cN^\Phi &=  \frac{i}{(2\pi)^{d}} \int_0^t \iint \partial^{\alpha} \hw_2(q-p) e^{-i\Phi} \widehat{\cP}_\gamma (s,k-p,\ell-k+p)\widehat{\mu}(s,k-p-\ell,p)  \, d\ell \, dp \, ds, \\
	\cN^\Psi &= -\frac{i}{(2\pi)^d} \int_0^t \iint \partial^{\alpha} \hw_2(q-p) e^{-i\Psi}  \widehat{\cP}_\gamma (s,-p+\ell,p) \widehat{\mu}(s,k-p,p-\ell) \,  d\ell \, dp \, ds,
\end{align*}
with the same phase functions $\Phi, \Psi$ as defined in \eqref{defPhi}. 
Then, observe that $\cI, \cL^X, \cN^\Phi$ and $\cN^\Psi$ are identical to that of $\Frho^{(0)}, \Frho^X, \Frho^\Phi$ and $\Frho^\Psi$ defined as in \eqref{defrho}, up to the appearance of $\partial^{\alpha} \hw_2(q-p)$, which is harmless, as it is bounded by $\epsilon_1$ in $C^N$ norm. The proposition thus follows. 
\end{proof}

\subsection{Proof of Proposition \ref{prop-iter}}

Proposition \ref{prop-iter} now follows straightforwardly. Indeed, in view of the definition of $\zeta_\rho(t)$ in \eqref{defzeta} and the results from Propositions \ref{prop-density} and \ref{prop-exchange}, we have obtained 
$$ \zeta_\rho(t) \le C_0 \epsilon_0 + C_1\epsilon_1 \zeta_\rho(t)+ C_2 \zeta(t)^2.$$
Taking $\epsilon_1$ sufficiently small to absorb the linear term $\zeta_\rho(t)$ to the left, we thus obtain $\zeta_\rho(t) \le C_0' \epsilon_0 + C_2' \zeta(t)^2$ for some universal constants $C_0', C_2'$. On the other hand, in view of the definition of $\zeta_\mu(t)$ in \eqref{defzeta} and the results from Proposition \ref{prop-profile}, we obtain 
$$ \zeta_\mu(t) \le C_0 \epsilon_0 + C_1 \zeta_\rho(t)+ C_2 \zeta(t)^2.$$
This, together with $\zeta_\rho(t) \le C_0' \epsilon_0 + C_2' \zeta(t)^2$, proves that $\zeta(t) = \zeta_\rho(t) + \zeta_\mu(t) \le C_0''\epsilon_0 + C_2'' \zeta(t)^2$, completing the proof of Proposition \ref{prop-iter}. 

\section{Proof of the main results}

In this section, we prove the main results stated in Theorem \ref{thm:main}. Indeed, the main results proven in Section \ref{sec-iter}, see Proposition \ref{prop-iter}, provide the boundedness of the iterative norm $\zeta(t) \le 2C_0 \epsilon_0$. In particular, in view of \eqref{defzeta} and \eqref{rhoXchange}, we obtain 
	\begin{equation}\label{est:decay-Fourier}
		| \Frho_{\gamma}(t,k)|
		+ |\widehat{\cP}_{\gamma}(t,k-p,p) |
		\lesssim \epsilon_0 \langle k\rangle^{-\sigma} \langle kt \rangle^{-N} \langle t \rangle^{\delta}
	\end{equation}
uniformly for all $t\ge 0$ and $k,p \in \mathbb{R}^d$. Applying the classical Hausdorff--Young inequality, for $p \in [2,\infty]$, we obtain  	\[
		\| \partial_x^n \rho_{\gamma}(t) \|_{L^p} 
		\lesssim \left( \int_{\RR^d} |k|^{np'} |\Frho(t,k)|^{p'} \, dk \right)^{1/p'}
		\lesssim \epsilon_0 \langle t \rangle^{-n-d(1-1/p)+\delta}
	\]
	for any integers $n < \min \{N,\sigma\}-d$, where $1/p +1/p'=1$, yielding the decay in the physical space as stated in the main theorem. Next, we establish the scattering theory in the Hilbert--Schmidt space. Indeed, using the profile equation \eqref{eqn:profile}, we bound 
\begin{equation}\label{dtmkp}
	\begin{aligned}
		\frac{d}{dt} \| \Fmu(t) \|_{L^2_{k,p}}
		& \le \|a_{k,p} \widehat{\cP}_{\gamma} (t)\|_{L^2_{k,p}}
+ (2\pi)^{-d} \Big\| \int |\widehat{\cP}_{\gamma}(t,k, \ell-k) \widehat{\mu}(t,k-\ell, p)| \; d\ell \Big\|_{L^2_{k,p}}
\\&\qquad +(2\pi)^{-d} \Big\| \int |\widehat{\mu}(t,k,p-\ell) \widehat{\cP}_{\gamma}(t,\ell-p, p)| \; d\ell \Big\|_{L^2_{k,p}} .
	\end{aligned}
\end{equation}
Using the decay estimate \eqref{est:decay-Fourier} and \eqref{dervakp}, we bound 
$$
\begin{aligned}
\|a_{k,p} \widehat{\cP}_{\gamma} (t,k,p)\|_{L^2_{k,p}}^2 
&= \|a_{k-p,p} \widehat{\cP}_{\gamma} (t,k-p,p)\|_{L^2_{k,p}}^2 
\\&\lesssim \epsilon_0^2\int |k|^2\langle k\rangle^{-2\sigma}\langle kt \rangle^{-2N} \langle t \rangle^{2\delta} \min\{ \langle k-p\rangle, \langle p\rangle\}^{-4n_1} \; dp \, d\ell
\\&\lesssim \epsilon_0^2 \langle t\rangle^{-2-d+2\delta},
\end{aligned}
$$
in which we have used the prefactor $|k|$ in the estimate \eqref{dervakp} of $a_{k-p,p}$, which leads to the extra decay of order $t^{-1}$ as stated. 
On the other hand, we bound 
$$\begin{aligned}
\Big|\int |\widehat{\cP}_{\gamma}(t,k, \ell-k) \widehat{\mu}(t,k-\ell, p)| \; d\ell \Big|^2
& \le \|\widehat{\cP}_{\gamma}(t,k, \ell-k)\|_{L^\infty_kL^1_\ell} \int |\widehat{\cP}_{\gamma}(t,k, \ell-k) \widehat{\mu}(t,k-\ell, p)|^2 \; d\ell, 
\end{aligned}
 $$
which yields 
$$\begin{aligned}
\Big\|\int |\widehat{\cP}_{\gamma}(t,k, \ell-k) \widehat{\mu}(t,k-\ell, p)| \; d\ell \Big\|^2_{L^2_{k,p}}
& \lesssim \|\widehat{\cP}_{\gamma}(t,k, \ell-k)\|_{L^\infty_kL^1_\ell}^2  \|\widehat{\mu}(t)\|^2_{L^2_{k,p}}. 
\end{aligned}
 $$
Using again \eqref{est:decay-Fourier}, we bound $\|\widehat{\cP}_{\gamma}(t,k, \ell-k)\|_{L^1_\ell}  \lesssim \epsilon_0\langle t\rangle^{-d+\delta}$. Finally, the last integral term in \eqref{dtmkp} is estimated similarly. As a result, we obtain 
\begin{equation}
		\frac{d}{dt} \| \Fmu(t) \|_{L^2_{k,p}} \lesssim \epsilon_0 \langle t \rangle^{-1-d/2+\delta} + \epsilon_0 \langle t \rangle^{-d+\delta} \| \Fmu(t) \|_{L^2_{k,p}},
\end{equation}
which in particular yields $\| \Fmu(t) \|_{L^2_{k,p}} \lesssim \epsilon_0$ by using the Gronwall lemma. Next, for any $t_1 \ge t_2 \ge 0$, we repeat the above analysis to obtain  
	\[
 \| \Fmu(t_1) - \Fmu(t_2) \|_{L^2_{k,p}} \lesssim \epsilon_0 \int_{t_2}^{t_1} (\langle s \rangle^{-1-d/2 +\delta} + \langle t \rangle^{-d+\delta}) \, ds \lesssim \epsilon_0 \langle t_2\rangle^{-d/2+\delta} .  
 	\]
Note that $\|\cdot \|_{ \mathrm{HS}} = \|\cdot \|_{L^{2}_{k,p}}$. That is, $\{\mu(t)\}_{t\ge 0}$ is a Cauchy sequence, and hence, has a unique limit $\gamma_\infty$, as $t\to \infty$, in  the Hilbert--Schmidt spaces with kernel $\gamma_\infty \in L^2_{x,y}$. In addition, recalling \eqref{defmu}, we obtain 
	\[
		\| \gamma(t)  -e^{itH_\infty} \gamma_{\infty}e^{-itH_\infty}  \|_{\mathrm{HS}} \lesssim \epsilon_0 \langle t \rangle^{-d/2 + \delta},
	\]
	which completes the proof of \Cref{thm:main}.

\appendix 

\section{Oscillatory integral}\label{sec:osc}
In this section, we recall the following classical lemma to derive dispersive estimates from a non-degerate oscillatory phase, see, e.g. \cite{Farah2012, HKNR5}. 

\begin{lemma}[Non-stationary phase]\label{lem:IBP}
Let $\Phi(p)$ be a $C^{N+1}$ function so that $|\nabla_p\Phi|\ge \lambda $ for some $\lambda>0$. Suppose that there is a universal constant $C_0$, independent of $\lambda$ and $p$, so that 
\begin{equation}\label{phase-cond}
|\partial_p^\alpha \Phi(p)| \le C_0 |\nabla_p\Phi(p)| , \qquad 2\le |\alpha|\le N+1, 
\end{equation}
uniformly in $p\in \RR^d$. Then, for any $F \in W^{N,1}$, there holds
	\[
		\left| \int e^{-i \Phi(p)} F(p) \, dp \right| \lesssim \langle \lambda \rangle^{-N} \| F \|_{W^{N,1}}.
	\]
\end{lemma}
\begin{proof} We only need to consider the case when $|\lambda |\ge 1$, since otherwise it is direct. Observe that $Le^{-i\Phi(p)} = e^{-i\Phi(p)}$ where $L$ and its adjoint $L^{\ast}$ are given by
	\[
		Lf =- \frac{\nabla \Phi(p)}{i|\nabla\Phi(p)|^2} \cdot \nabla f, \qquad L^{\ast} g = - \frac{1}{i} \nabla \cdot \left( \frac{\nabla \Phi}{|\nabla \Phi|^2} g \right). 
	\]
	Integrating by parts for $N$ times, we get
	\begin{align*}
		\left| \int e^{-i\Phi(p)} F(p) \, dp  \right| 
		= \left| \int e^{-i\Phi(p)} (L^{\ast})^N F(p) \, dp \right|
		\lesssim \| (L^{\ast})^N F \|_{L^1}.
	\end{align*}
In view of \eqref{phase-cond}, we may write $(L^{\ast})^N = \sum_{|\alpha|\le N}a_\alpha^{(N)}\partial^\alpha_p $, where $|a_\alpha^{(N)}|\lesssim \langle \nabla_p \Phi\rangle^{-N}$. This yields  
	\[
		\| (L^{\ast})^N F \|_{L^1} \lesssim \langle \lambda \rangle^{-N} \sum_{|\alpha| \le N} \| \partial^{\alpha}_p F \|_{L^1},
	\]
	concluding the proof.
\end{proof}

\bibliographystyle{abbrv}

\end{document}